\documentclass[dvips]{imsart}

\RequirePackage[OT1]{fontenc}
\RequirePackage{amsthm,amsmath}
\RequirePackage[numbers]{natbib}
\RequirePackage[colorlinks,citecolor=blue,urlcolor=blue]{hyperref}
\usepackage{graphicx}
\usepackage{amssymb}
\usepackage{epstopdf}

\startlocaldefs

\newcommand{\xm}{{X_{mix}}}
\newcommand{\vm}{{V_{mix}}}

\numberwithin{equation}{section}
\theoremstyle{plain}
\newtheorem{example}{Example}%[section] 
\newtheorem{corollary}{Corollary}[section] 
\newtheorem{definition}{Definition}[section] 
\newtheorem{lemma}{Lemma}[section] 
\newtheorem{thm}{Theorem}[section]

\newcounter{foo}

\endlocaldefs

\begin{document}

\begin{frontmatter}

% "Title of the Paper"
\title{Computational information geometry: theory and practice}
%\thankstext{t1}{This is an original survey paper}
\runtitle{Computational information geometry}

% indicate corresponding author with \corref{}
% \author{\fnms{John} \snm{Smith}\thanksref{t2}\corref{}\ead[label=e1]{smith@foo.com}\ead[label=e2,url]{www.foo.com}}
% \thankstext{t2}{Thanks to somebody} 
% \address{line 1\\ line 2\\ \printead{e1}\\ \printead{e2}}

\author{\fnms{Karim} \snm{Anaya-Izquierdo}\ead[label=e1]{karim.anaya@lshtm.ac.uk}}
\address{London School of Hygiene and Tropical Medicine,  Keppel Street, London   WC1E 7HT, UK \printead{e1}}
%\and
%\author{\fnms{???} \snm{???}\ead[label=e2]{???}}
%\address{\printead{e2}}

\and
\author{\fnms{Frank} \snm{Critchley}\ead[label=e2]{f.critchley@open.ac.uk}}
\address{Department of Mathematics and Statistics,          The Open University, Walton Hall, Milton Keynes, Buckinghamshire. MK7 6AA, UK \\  \printead{e2}}

\and
\author{\fnms{Paul } \snm{Marriott}\ead[label=e3]{pmarriot@uwaterloo.ca}}
\address{Department of Statistics and Actuarial Science,    University of Waterloo, 200 University Avenue West, Waterloo, Ontario, Canada N2L 3G1 \\\printead{e3}}

\and
\author{\fnms{Paul W.} \snm{Vos}\ead[label=e4]{vosp@ecu.edu}}
\address{Department of Biostatistics,    East Carolina University  2435C Health Sciences Building, Greenville,  NC 27858-4353 USA \\ \printead{e4}}

\runauthor{Anaya-Izquierdo et al. }

\begin{abstract}
This paper lays the foundations for a unified framework for numerically and computationally applying methods drawn from a range of currently distinct geometrical approaches to statistical modelling. In so doing, it extends information geometry from a manifold based approach to one where the simplex is the fundamental geometrical object, thereby allowing applications to models which do not have a fixed dimension or support. Finally, it starts to
build a computational framework which will act as a proxy for the Ôspace of all distributionsÕ that can be used, in particular, to investigate model selection and model uncertainty. A varied set of substantive running examples is used to illustrate theoretical and practical aspects of the discussion. Further developments are briefly indicated. 
\end{abstract}



% history:
% \received{\smonth{1} \syear{0000}}

%\tableofcontents

\end{frontmatter}

\section{Introduction}  \label{Introduction}

The application of geometry to statistical theory and practice has produced a number of different approaches and this paper will involve  three of these. The first  is  the  application of differential geometry to statistics, which is often called  information geometry. It largely focuses  on typically multivariate, invariant  and higher-order  asymptotic results  in full and  curved exponential families through the use of differential geometry and tensor analysis; key references include  \cite{Amar:1990}, \cite{Barn:Cox:1989},  \cite{Barn:Cox:1994},  \cite{Murr:Rice:1993}  and \cite{Kass:Vos:1997}.  Also included in this approach are consideration of curvature, dimension reduction and information loss, see \cite{Efro:1975} and \cite{Marr:Vos:On:2004}. 
The  second important,  but completely separate, approach is in the inferentially demanding area of mixture modelling, a major  highlight being  found in   \cite{Lind:mixt:1995} where convex geometry is shown to give great insight into the fundamental problems of inference in  these models and to help in the design of corresponding algorithms.  The third approach is the geometric study of graphical models, contingency tables, (hierarchical) log-linear models, and related topics involving the geometry of extended exponential families.   Important results with close connections to the approach in this paper can be found in \cite{Rina:Fein:Zhou:2009}  and  \cite{Rina:Fein:2011},     while the wider field of algebraic statistics is well-reviewed in  \cite{Pist:Rico:Wynn:2000} and \cite{Gibilisco:Wynn:Alge:2010}. 
%\footnote{add to bib Algebraic and Geometric Methods in Statistics Paolo \textsc{Gibilisco}, Eva \textsc{Riccomagno}, Maria Piera \textsc{Rogantin}, and Henry P. Wynn \textsc{(Eds.)}.  New York, NY: Cambridge University Press, 2010. }. 

This paper  has the following  four objectives:  (1) to use the tool  of  the extended multinomial distribution (see \cite{Brow:1986}, \cite{Rina:Fein:Zhou:2009}, \cite{Rina:Fein:2011} and   \cite{Csis:Matu:2005}) to construct a  framework which unifies all of the above geometric approaches; in particular, to show explicitly the links between information geometry, extended exponential families and Lindsay's mixture geometry, (2) to show how this unifying framework provides a natural home for numerically implementing algorithms based on the geometries described above,  (3)  to extend the results of information geometry from the traditional manifold based approach to models which do not have a fixed dimension or support, and (4) to start to  build a computational framework which will act as a proxy for the `space of all distributions' which can be used, in particular,  to investigate model selection and model uncertainty.  This paper lays the conceptual foundations for these goals, with more detailed developments to be found in later work. We call this numerical way of implementing geometric theory in statistics computational information geometry. No confusion should arise from the fact that the same name is given to a cognate, but distinct, topic in machine learning: see for example 
\cite{2009-CompInfoGeomMeaningDistance-CSL}.

In practice, a single statistical problem can involve more than one of the above geometries -- potentially all three -- this plurality being handled naturally in our unifying framework. Indeed, we use a varied set of substantive running examples to illustrate theoretical and practical aspects of the development.   Examples \ref{binomial example} and \ref{tripod example} (Section  \ref{examples defined}) are chosen to illustrate computational information geometric issues in mixture models.  Example \ref{censored exponential  example} shows issues in full and curved exponential families, while Example \ref{logistic regression} looks at the geometry of logistic regression.  To aid with visualisation additional low dimensional multinomial models are also introduced in the body of the paper.

The key idea of this paper  is to represent  statistical models --   sample spaces, together with  probability distributions on them --  and associated inference problems, inside adequately large but  finite dimensional  spaces.  In these  embedding spaces  the building blocks of the three geometries described above  can be  numerically computed  explicitly and the results  used for algorithm development.   In \S \ref{section on Discretisation}  and   \S \ref{Discussion and further work} we reflect on the generality of working in this finite, discrete framework even with models for continuous random variables.

Accordingly, after a possible initial discretisation,
the space of all distributions for the random variable of interest can be identified with the simplex, \begin{equation}\label{definition of extended multinomial}
\Delta^k:=\left\{\bf{\pi}=(\pi_0, \pi_1,\ldots,\pi_k)^\top\,:\, \pi_i\geq 0\,,\,\sum_{i=0}^k\pi_i=  1\right\},
\end{equation} together with a unique label for each vertex, representing the random variable. Modulo  discretisation,   this structure   therefore acts as a universal model.   Clearly, the multinomial family on $k+1$ categories  can be identified with the relative interior of this space, $int(\Delta^k)$, while the extended family allows the possibility of distributions with different support sets.

 The starting point for much of statistical inference is a working model for observed data comprising a set of  distributions on a sample space.   A  working model ${\cal M}$ can be represented by a subset of  $\Delta^k$ and may be specified by an explicit parameterisation, such as Example \ref{censored exponential  example},  or as the solution of a set of equations, such as Example  \ref{tripod example}.  Computational information geometry   explicitly uses the information geometry of $\Delta^k$ to numerically  compute statistically important features of ${\cal M}$. These features include properties of the likelihood, which can be nontrivial in many of the examples considered here,  the adequacy of first order asymptotic methods -- notably, via higher order asymptotic expansions -- curvature based dimension reduction and
inference in mixture models.

\subsection{Examples}\label{examples defined}

 For ease of reference the main examples considered in this paper  are briefly described here, together with the main points which they illustrate.

 \begin{example}\label{binomial example}{\bf Mixture of binomial distributions}
 This example comes from  \cite{Kupp:Hase:1978} where the authors state  that `simple one-parameter binomial and Poisson models generally provide poor fits to this type of binary data', and therefore it is of interest  to look  in a `neighbourhood' of these models. The extended multinomial space is a natural place to define such a `neighbourhood' and a new computational algorithm defined in \S \ref{Mixture Inference} is used for inference. 
 \end{example}

%\begin{example}\label{binomial example}
%Consider the  data  discussed in \cite{Kupp:Hase:1978} shown in part in Table \ref{example1table1}. The authors state  that `simple one-parameter binomial and Poisson models generally provide poor fits to this type of binary data', and therefore it is of interest  to look  in a `neighbourhood' of these models. In particular, mixture models are  of interest scientifically since the data concerns frequency of implanted foetuses  in   laboratory animals,  and it could be expected that there is underlying clustering. 
% % Requires the booktabs if the memoir class is not being used
%
% 
% \begin{table*}
%\caption{Observed frequencies of number of dead implants}
%\label{example1table1}
%  \begin{tabular}{|c|ccccccc|} \hline
%Number of dead implants & 0 & 1 & 2 & 3 & 4 & 5 & $>5$ \\
%Frequency & 214 &  154 &   83  & 34  & 25 &    9  &  5  \\ \hline
%         \end{tabular}
%\end{table*}
% \end{example}
% 

 \begin{example}\label{censored exponential  example} {\bf Censored exponential} This example looks at a continuous response variable -- a censored survival time. Section \ref{section on Discretisation} considers  applying the results of computational information geometry to models for continuous random variables while Theorems \ref{general information loss} and  \ref{exponential information loss theorem} show how this can be done with negligible loss for inference. In this case also  results on curvature based dimension reduction are illustrated.
 \end{example}

\begin{example}\label{logistic regression} {\bf Logistic regression}  This is a full exponential family that   lies in a very high dimensional  simplex when considered as a model for the joint distribution of  $N$ binary response variates. In this example,  both the existence  of the maximum likelihood estimate (see \cite{Rina:Fein:Zhou:2009} and \cite{Rina:Fein:2011})  and higher order approximations to sampling  distributions are considered. 
%
%
%
% A design matrix  $X$ defines  a $D$-dimensional $+1$-affine subset and changing the explanatory variates  changes the direction of this low dimensional space inside the space of joint distributions. 
%
%
%
%
%
%Consider a $N \times D$ design matrix $X$ with $N$ samples and a binary response $t \in \{0, 1\}^N$. Let 
%$s(x) = \log\left( \frac{x}{ 1- x} \right)$ and so $s^{-1}(x) = \frac{\exp(x)}{1+\exp(x)}$ and the  logistic regression  model is given by 
%$$P(T_i=1)= s^{-1}(\beta^T X^T_{i, })$$ where $X_{i, }$ is the $i^{th}$ row of $X$ the design matrix.
%This is a full exponential family that   lies in the $2^N-1$ simplex when considered a model for the joint distribution of the $N$ binary response variates. A design matrix  $X$ defines  a $D$-dimensional $+1$-affine subset and changing the explanatory variates  changes the direction of this low dimensional space inside the space of joint distributions. 
\end{example}

\begin{example}\label{tripod example} {\bf Tripod model}
The tripod  example is discussed in \cite{Zwie:Smith:2011b} and \cite{Zwie:Smith:2011a}.  The  directed graph is shown in Fig.~\ref{hiddenvarfig}, where  there are  binary variables $X_i$,  $i=1, 2, 3$, on each of the terminal nodes these being assumed independent given the binary variable at the internal node $H$. In the model, it is assumed  $H$ is hidden (i.e. not observed)  so the model is a mixture of members of  an exponential family. Despite the model's  apparent simplicity, the mixture structure can generate multiple modes in the likelihood,  illustrating difficult identification issues. 
  \begin{figure}[htbp] %  figure placement: here, top, bottom, or page
   \centering
   \includegraphics[width=2in]{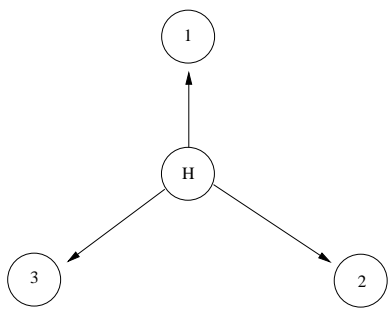}
   \caption{Graph for Tripod model}
   \label{hiddenvarfig}
\end{figure}
\end{example}

\subsection{Discretisation}\label{section on Discretisation}

The  approach taken in this paper is  inherently discrete and  finite. Sometimes, this is with no loss at all, the models used involve only such random variables.  In general,  suitable finite partitions  of the sample space can be  used,  for which  an appropriate theory is developed.   While this is clearly not the most general case mathematically speaking  (an  equivalence relation  being thereby induced), it does provide an excellent foundation on which to  construct a computational  theory.  Furthermore, since real world measurements can only be made to a fixed precision all models can  -- arguably, should --   be thought of as fundamentally categorical. The relevant question for a computational theory is then:  what is the effect on  the inferential objects of interest  of a particular  selection  of such categories?   This is looked at in Theorem  \ref{general information loss} and  \ref{exponential information loss theorem}.

%
% Indeed it has been argued, \cite{pit;some;1979}
%\begin{quote}
%``$\cdots$ statistics being essentially a branch of applied mathematics, we should be guided in our choices of principles and methods by the practical applications. All actual sample spaces are discrete, and all observable random variables have discrete distributions. The continuous distribution is a mathematical construction, suitable for mathematical treatment, but not practically observable."
%\end{quote}

\setcounter{foo}{\value{example}}
\setcounter {example} {1}
\begin{example}[continued]  Here   the data taken from \cite{Hand:Daly:Lunn:1994}, while being treated as continuous, is only recorded at integer number of days. Thus as far as any statistical analysis that can be carried out is concerned there is literally   zero  loss in treating it as sparse categorical.  For Figs.~\ref{example2fig2.pdf} and \ref{example2fig4.pdf}  a further level of coarseness was added  by selecting bins of size $4$ days. As can be seen from the likelihood plot, Fig.~\ref{example2fig2.pdf},  there is effectively no inferential loss in such a choice.
\end{example}

\setcounter{example}{\value{foo}}

\subsection{Structure of paper}

The paper is structured as follows. Section \ref{Geometry of extended multinomial distrubution} looks at the information geometry of $\Delta^k$.  It shows the geometry to be both   explicit  and tractable.  In particular, the way that  global geometry determines the relationship between the natural and mean parameters of exponential families is discussed in \S \ref{Finite dimensional case}. The Fisher information is also key and results on its spectrum are found in  \S \ref{Spectrum of Fisher Information}, while the shape of the likelihood function is discussed in \S \ref{Likelihood in simplex}.  Section \ref{Duality and closure} looks at  the importance of  understanding  the closure of $\Delta^k$,  and of exponential families embedded in $\Delta^k$, where we consider the computation of limit points and the corresponding behaviour  of maximum likelihood estimates.  Direct applications of the numerical approach are discussed  in Section \ref{Classical information geometry}. Issues considered include:   using  higher order asymptotic methods, such as Edgeworth and saddlepoint expansions  and, also,  dimension reduction and information loss.   Section \ref{Mixture Inference} looks at the way that the mixture  geometry of \cite{Lind:mixt:1995}   fits naturally into the computational information geometry framework. In this section, Examples  \ref{binomial example} and \ref{tripod example} show the utility of the methods. Again the issue of dimension, this time in the $-1$-geometry, comes to the fore. Throughout, proofs and more technical discussions are found in the appendices.

\section{Geometry of extended multinomial distribution}\label{Geometry of extended multinomial distrubution}

The key idea behind  computational information geometry is that models can be embedded in a computationally tractable space with little loss to the inferential problem of interest. Information geometry is constructed from two different affine geometries related in a  non-linear way via duality and the Fisher information, see \cite{Amar:1990} or \cite{Kass:Vos:1997}.  In the full exponential family context, one affine structure (the so-called $+1$ structure) is defined by the natural parameterization, the second (the $-1$ structure) by the mean  parameterization.  The closure of exponential families has been studied by \cite{Barn:1978}, \cite{Brow:1986}, \cite{Laur:1996} and \cite{Rina:2006} in the finite dimensional case and  by \cite{Csis:Matu:2005} in the infinite dimensional case.    One important difference in the approach taken here is that limits of families of distributions, rather than pointwise limits, are central.

This paper constructs  a theory of information geometry following that  introduced by   \cite{Amar:1990} via the affine space construction introduced by \cite{Murr:Rice:1993}  and extended by \cite{Marr:on:2002}. Since this paper  concentrates on categorical random variables, the following definitions are appropriate. Consider   a  finite set of disjoint categories or bins $\mathcal B=\{B_i\}_{i\in A}$. Any distribution over this finite set of categories is defined by a set $\{ \pi_i\}_{i \in A}$ which defines the corresponding probabilities.

\begin{definition}\label{def1}The $-1$-affine space structure over distributions on  $\mathcal B:=\{B_i\}_{i\in A}$   is $\left(\xm, \vm, +\right)$  where
$$\xm= \left\{ \{ x_i\}_{i \in A}  | \sum_{i \in A} x_i =1  \right\}, \vm= \left\{   \{v_i\}_{i \in A} | \sum_{i \in A} v_i =0  \right\} $$ and the addition operator $+$  is the usual addition of  sequences.
\end{definition}
In Definition \ref{def1} the space of (discretised) distributions  is  a $-1$-convex subspace of the affine space  $ \left(\xm, \vm, +  \right) $.  A similar affine structure for the $+1$-geometry, once the support has been fixed, can be derived from the definitions in \cite{Murr:Rice:1993}.

The extended multinomial family, over $k+1$-categories, characterized by the closed simplex of probabilities $\Delta^k$ defined in (\ref{definition of extended multinomial}) will be the  computationally tractable space. For these families, the  $\pm 1$ dual affine geometries are explicit,  the only `hard' computational tasks being the non-linear mapping between convex subsets of affine spaces and the computation of the mixed parameterization, as defined in  \cite{Barn:Blae:1983}. Furthermore,  the Fisher information and its inverse are explicit and, perhaps more relevantly due to its potentially high order (the dimension of the simplex)  and non-constant rank, there are good ways of understanding and bounding its spectrum, as shown in \S \ref{Spectrum of Fisher Information}.

It is important to clarify why the closed extended multinomial distribution is used. First, in many examples the data is sparse in the sense that the sample size is much smaller than $k+1$, the number of categories, so that the likelihood, both in the multinomial and sometimes in the embedded models, is typically maximized on the boundary.  Second, it will  be shown that the  global shape of the likelihood function is determined by boundary behaviour.  Third, first order asymptotic approximations are rarely uniform across $\Delta^k$ and the higher
order asymptotic expansions of computational information geometry can indicate when the boundary is inferentially
relevant.  Finally, the link between information geometry and  Lindsay's mixture geometries is defined by using   the boundary of  $\Delta^k$.

The probability simplex, and sub-models embedded in it,  have been extensively studied in the  geometric approach to graphical models, see \cite{Rina:Fein:Zhou:2009}, \cite{Rina:Fein:2011}. In this literature, other sampling schemes than the multinomial are also studied, boundary issues again being shown to have great importance. One of the important new features here is the application of the full information geometry machinery  to these models.

\subsection{Geometry of extended trinomial distribution}\label{Finite dimensional case}

 To illustrate the information geometry of the extended multinomial distribution, the trinomial case is now described explicitly.  The general case in fact will follow by obvious extensions, and shown later (Section \ref{finite duality}), unless  the dimension is so large that numerically evaluating sums becomes impractical, see \cite{Geig:Heck:Strat:2001}.

\begin{example}\label{trinomial example}
 An explicit example of the information geometry of the extended trinomial model   is shown in  Fig.~\ref{dualityplot}.  The closed simplex in panel (a) represents the set  of multinomial distributions with bin  probabilities  $(\pi_0, \pi_1, \pi_2)$ where $\pi_i \ge 0$.

 In this example, the vector $b^T=(1,2,3)$  was chosen,  and  the parallel lines in panel (a) are  level sets of the mean of $b^TX$, where $X$ is the trinomial random variable. In the terminology of classical information geometry, these are $-1$-geodesics, and  it is immediate that they extend to the boundary in a very natural way.  These lines lie in the (tangent) direction $a$ which satisfies  
 $\sum_{k=0}^2 a_k =0$, and $ \sum_{k=0}^2 a_kb_k =0.$  These lines are also  shown in panel (b), but now in the $+1$ (or natural) parameterization and so are non-linear. Note that  the single line, labelled  by the mean value equalling $2$, corresponds to the $-1$ geodesic passing  through the vertex at $(1,0)$ in panel (a).

\begin{figure}[h] %  figure placement: here, top, bottom, or page
   \centering
 %   \figurebox{20pc}{20pc}{}[dualityplot.eps]
   \includegraphics[width=3in]{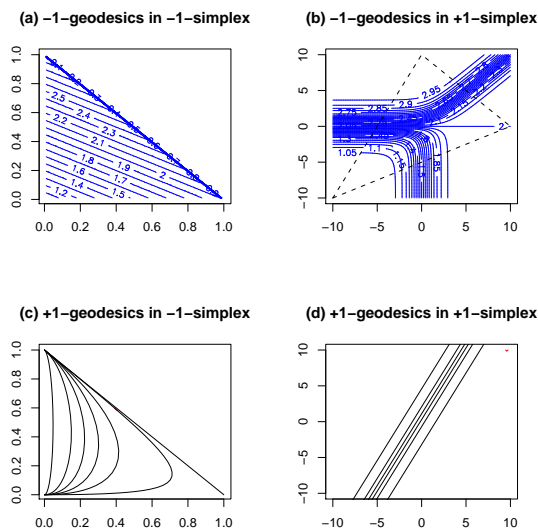} \\
   \caption{The information geometry of the extended trinomial model}
   \label{dualityplot}
\end{figure}

Panel (d)  shows the relative interior  of the extended trinomial in the natural affine parameterization. The straight  lines represent one dimensional full exponential families with  probabilities of the form 
$$
\left(  \frac{ \pi_0 \exp(\theta b_0)   }{ \sum_{k=0}^2 \pi_k \exp(\theta b_k)    }, \frac{ \pi_1 \exp(\theta b_1)   }{ \sum_{k=0}^2 \pi_k \exp(\theta b_k)    },\frac{ \pi_2 \exp(\theta b_2)   }{ \sum_{k=0}^2 \pi_k \exp(\theta b_k)    }  \right),
$$ each $\pi_k > 0$. These are $+1$ -geodesics in the direction $b$ through the base-point $(\pi_0,\pi_1, \pi_2)$ and, by the strict positivity of the exponential function, their image in panel (c)   lie strictly in the interior of the simplex.   It is a standard result that these $+1$ parallel lines are everywhere orthogonal, with respect to the  metric defined by the Fisher information matrix,  to the $-1$-parallel lines shown in panels (a) and (b). Each of these parallel lines can be found by moving  the base-point by 
$$ (\pi_0,\pi_1, \pi_2)    \mapsto (\pi_0,\pi_1, \pi_2)+  \sigma(a_0, a_1, a_2),  $$ $\sum_{k=0}^2 a_k=0$, 
where $\sigma$ is restricted so that all components remain non-negative,  \cite{Marr:on:2002}.

The key step in understanding the simplicial  nature of the $+1$-geometry is to see how the limits of the $+1$-parallel lines are connected to the boundary of the simplex.  This is made  clear  in panel (c),  where the $+1$-geodesics are plotted in the $-1$-affine parameters as curves. As $\sigma$ changes the limits of the curves clearly exist and lie on  the boundary of the simplex. The closure of the $+1$-representation  multinomial  is defined to make these continuous limits defined ``at infinity'' in the $+1$-parameters and is shown schematically as the dotted triangle in panel b. 

% \begin{figure}[htbp] %  figure placement: here, top, bottom, or page
%   \includegraphics[width=3in]{art/simplexplotfields.eps}
%    \caption{Trinomial model: attaching the geodesics to the boundaries}
%    \label{simplexplotfields.eps}
% \end{figure}
% 
% 
% The plots in Fig.~\ref{dualityplot} are constructed by keeping a tangent direction fixed and keeping the  geodesics parallel. An alternative way of visualizing the geometry is shown in Fig.~\ref{simplexplotfields.eps} where a a fixed point has been chosen  and different tangent directions are selected. This is a form of `polar-coordinate' system.  In the lefthand panel the red lines are lines which have the property that they are both $+1$ and $-1$-geodesics. These correspond to the tangent directions $a$ where two of the components are equal, Section \ref{Computing limits in exponential families} looks in more detail how the rank structure of the tangent direction determines the limit points of a geodesic on the boundary. The cone of black lines in  Fig.~\ref{simplexplotfields.eps} correspond to a set of tangent directions where the ranks of the tangent vectors are fixed. Its is clear that they limit points correspond to the vertex corresponding to the maximum and minimum values.   In the righthand panel the same cone is shown but also 
 
\end{example}

%
%The discussion in Example \ref{trinomial example} is representative of the general extended multinomial family.  Appendix 1 shows,    for the general extended multinomial family, a way of formally defining the underlying topology for both the $\pm 1$-structures.  It is shown that the topology of the simplex is appropriate for both, and that the likelihood function is continuous in both cases. The fact that the topology in the $-1$-affine space of the extended multinomial is simplicial is immediate from its convex subset structure within the $-1$-affine space. The fact that the same topology can be defined  for the  $+1$-structure comes from the observation that in the interior of each sub-simplex  there is a diffeomorphism between the $+1$ and $-1$ parameters. The topology in the $+1$ structure is then inherited from that of the $-1$ by using the weak topology.  In other words, the $+1$ components are ``glued'' together at infinity in the same topological way as the $-1$ simplex.
%

%{\bf Non-linear transformation needed for higher order asymptotics determined by topology, mixed parameterisation, foliations }

\subsection{Spectrum of Fisher Information}\label{Spectrum of Fisher Information}

The material  above looks explicitly at the $\pm 1$-affine geometries of \cite{Amar:1990} while this section concentrates on the third part of Amari's structure, i.e. the Fisher information or $0$-geometry. In any multinomial model,  the Fisher information matrix and its inverse are explicit. Indeed, the $0$-geodesics  and the corresponding geodesic distance are also explicit, see 
\cite{Amar:1990} or \cite{Kass:Vos:1997}. However, since   the simplex glues together  multinomial structures with different supports,  and the computational theory is in high dimensions, it is a fact that the Fisher information matrix can be arbitrarily  close to being singular. It is therefore of central interest that the spectral decomposition of the Fisher information itself has a very nice structure, as shown in this section.

\begin{example}\label{discretised normal}

Consider a multinomial distribution based on $81$ categories of equal width on $[-5, 5 ]$, where the probability associated to a  bin is proportional to that of the standard normal distribution for that bin. The Fisher information for this model is an $80 \times 80$ matrix whose spectrum is shown in Fig.~\ref{fisherspectrum}. By inspection it can be seen that there are exponentially small eigenvalues, so that while the matrix is positive definite it is also arbitrarily close to being singular. Furthermore, it can be seen that the spectrum has the shape of a half-normal density function and that the eigenvalues seem to come in pairs. These facts are direct consequences of the following results.

\begin{figure}[htbp] %  figure placement: here, top, bottom, or page
   \centering
   \includegraphics[width=2in]{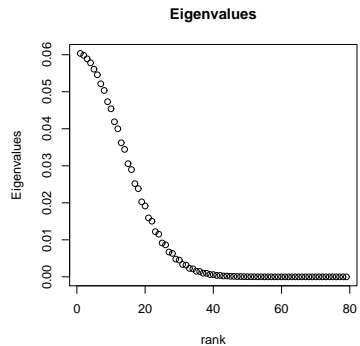}
   \caption{Spectrum of the  Fisher information matrix of  a discretised normal distribution}
   \label{fisherspectrum}
\end{figure}
\end{example}

With $\pi_{(0)}$  denoting the vector of all bin probabilities except $\pi_0$, the Fisher information matrix for the $+1$ parameters, written as a function of the probabilities,  is the sample size times 
$$I(\pi):= diag(\pi_{(0)}) - \pi_{(0)} \pi_{(0)}^T,$$ whose explicit spectral decomposition given, in all cases, in   Appendix 1, is an example of  interlacing eigenvalue results, (see for example  \cite{Horn:John:1985}, Chapter 4).     In particular, suppose $\{\pi_{i}\}_{i=1}^{k}$ comprises $g>1$
distinct values $\lambda_{1}> \cdots >\lambda_{g}>0$, $\lambda_{i}$ occuring
$m_{i}$ times, so that ${\textstyle\sum\nolimits_{i=1}^g} m_{i}=k$. Then, the spectrum of $I(\pi)$ comprises $g$ simple eigenvalues
$\{\widetilde{\lambda}_{i}\}_{i=1}^{g}$ satisfying
\begin{equation}\label{interleaving result}
\lambda_1> \tilde \lambda_1 > \dots > \lambda_g > \tilde \lambda_g \ge 0, 
\end{equation} together, if $g<k$, with $\{\lambda_{i}:m_{i}>1\}$, each such $\lambda_{i}%
$\ having multiplicity $m_{i}-1$. Further, $\widetilde{\lambda}_{g}%
>0\Leftrightarrow\pi_{0}>0$\ while each $\widetilde{\lambda}_{i}$ $(i<g)$\ is typically (much) closer
to $\lambda_{i}$ than to $\lambda_{i+1}$, making it a near replicate of
$\lambda_{i}$.

In this way, the Fisher spectrum mimics key features of the bin probabilities. Of central importance, one or more eigenvalues are exponentially small if and only if the same is true of the bin probabilities, the Fisher information matrix being singular if and only if one or more of the
$\left\{  \pi_{i}\right\}  _{i=0}^{k}$ vanishes. Again, typically, two or more eigenvalues will be close when two or more corresponding bin probabilities
are. We see this in Example \ref{discretised normal} where, by symmetry of the distribution, the bin probabilities are
paired, so that $m_{i}=2$. The (decreasingly) ordered plot of the eigenvalues, Figure \ref{fisherspectrum},  then resembles two copies of
the half-density formed by folding at the mode.  These dominant
features are robust to which bin we omit in forming $\pi_{(0)}$ and to
asymmetric placing of the bins. 

\subsection{Likelihood in the simplex}\label{Likelihood in simplex}

Potentially high dimensional simplicial structures being  the natural spaces in which to base computational information geometry, a primary question is to  look at the way that the likelihood, or log-likelihood, behaves in them. First note two important issues: in  typical  applications  the sample size will be much smaller than the dimension of the simplex, while   the simplex contains sub-simplexes    with varying  support. These two statements mean that our standard intuition about the shape of the log-likelihood function will not hold.   In particular, the standard  $\chi^2$-approximation  to the distribution of the deviance does not hold.

It will be convenient to call the face of the simplex spanned by the vertices (bins)  having strictly positive counts   the observed face, and the face spanned by the complement of this set  the unobserved face.  In the $-1$-representation, the log-likelihood is strictly concave on the observed face, strictly decreasing in the normal direction from it to the unobserved face and, otherwise, constant. This is illustrated  -- a schematic representation of the quadrinomial case when there are two zeros in the vector of counts -- in Figure~\ref{testxfig}, the $-1$-flat subspaces being formalised in Theorem \ref{shape of likelihood}.

%Interestingly  though,  a completely different asymptotic argument, based on the dimension of the simplex, can give very good approximations, see see \cite{Anay:crit:marr:vos:Sparse:2011}.

\begin{figure}[htbp] %  figure placement: here, top, bottom, or page
   \centering
   \includegraphics[width=2in]{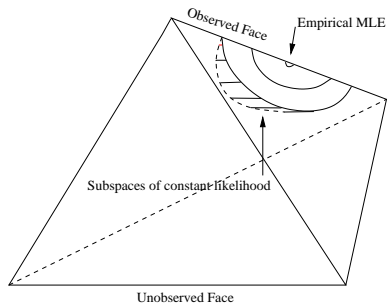}
   \caption{The shape of the likelihood in a simplex}
   \label{testxfig}
\end{figure}

The following theorem characterises the shape of the log-likelihood function in the $-1$-representation on the simplex. This function  is  concave, but not strictly concave,  so, the theorem characterises  where the lack of strict concavity comes from.  Being given by the function $\sum_{i\in\mathcal{P}} n_{i} \log\pi_{i}$, with the
constraints $\sum_{i \in\mathcal{P} \cup\mathcal{Z}} \pi_{i}=1$ and $\pi_{i}
\ge0$ it is immediate that the log-likelihood is constant on subsets defined by fixing $\pi_i \in \mathcal{P}$ and varying $\pi_i \in \mathcal{Z}$.    The decomposition presented in part (b) of the theorem shows that these subsets are, in fact, contained in -1-affine subspaces.

\begin{thm}\label{shape of likelihood} Let the observed counts be $\{n_i\}_{i=0}^K$ and define two subsets of the index set $\{0, \cdots, k\}$  by $\mathcal{P} =\{ i | n_i > 0 \}$ and $\mathcal{Z}= \{ i | n_i = 0 \}$. Let  $V_{\mathrm{mix}} =  \{ (v_0, \dots, v_k) | \sum v_{i} =0 \}$, and 
further define the  set $V^{0} \subset V_{\mathrm{mix}}$ by
$
\{ v \in V_{\mathrm{mix}} | v_{i}=0 \; \forall i \in P \}
$.

(a) The set  $V^{0}$ is a linear subspace of $V_{\mathrm{mix}}$. The log-likelihood is
constant on $-1$ affine subspaces of the form $\pi+ V^{0}.$

 (b)  Select $k^{*}
\in\mathcal{Z}$ and  consider the vector subspace of $V_{\mathrm{mix}}$ defined by
\[
V^{k^{*}}:= \left\{  v \in V_{\mathrm{mix}} | v_{i} =0 \mathrm{\; if\;} i
\in\mathcal{Z} \backslash\{ k^{*}\} \right\}  .
\]
 Then ${V_{mix}}$ can be decomposed as a direct sum of vector spaces
$
V_{\mathrm{mix}} =V^{0} \oplus V^{k^{*}}$.

\end{thm}

\begin{proof}
See Appendix.
\end{proof}

\section{Closure of exponential families}\label{Duality and closure}

This section shows  how the closure of  exponential families plays a role in the computational  geometry.  In \S \ref{finite duality}  the discussion of   \S   \ref{Finite dimensional case} is formalised and connected to the information geometric concept of duality.  Furthermore, in  \S  \ref{Computing limits in exponential families} Example  \ref{logistic regression}  is used to  illustrate  the fact that the way that the boundaries of the high dimensional simplex are attached to the model is of great importance for the behaviour of the likelihood and the  for distribution of important inferential statistics.  
 
%%%%%%%%%%%%%%%%%%%%%%%%%%%%%%%%%%%%%%%%%%%%%%%%%%%%%%%%%%%%%%%%%%%%%%%%%
%%%%%%%%%%%%%%%%%%%%%%%%%%%%%%%%%%%%%%%%%%%%%%%%%%%%%%%%%%%%%%%%%%%%%%%%%
%%%%%%%%%%%%%%%%%%%%%%%%%%%%%%%%%%%%%%%%%%%%%%%%%%%%%%%%%%%%%%%%%%%%%%%%%
%%%%%%%%%%%%%%%%%%%%%%%%%%%%%%%%%%%%%%%%%%%%%%%%%%%%%%%%%%%%%%%%%%%%%%%%%
%%%%%%%%%%%%%%%%%%%%%%%%%%%%%%%%%%%%%%%%%%%%%%%%%%%%%%%%%%%%%%%%%%%%%%%%%
\subsection{Duality}\label{finite duality}

One of the key aspects of information geometry is the relationship between the $+1$, $-1$ and Fisher metric or  $0$-geometric structures via the concept called duality. Following \cite{Amar:1990}  when the underlying geometric object is a manifold  the relationship between the $+1$ and $-1$ connections, denoted by $\nabla^{+1}$ and $\nabla^{-1},$ and the Fisher information is captured in the duality relationship which can be written in terms of the inner product at $\theta$, $\left<  , \right>_\theta$, and any vector fields $X,Y,Z$ via the equation
\begin{equation}\label{duality}
X\left< Y, Z\right>=   \left< \nabla^{+1}_XY, Z\right>+ \left< Y, \nabla^{-1}_XZ\right>.
\end{equation}
One consequence of this  relationship is the existence on exponential families of a so-called mixed parameterization of the form $(\theta, \mu)$, where $\theta$ is $+1$-affine and $\mu$ is $-1$-affine, their level sets being  Fisher orthogonal across the manifold: see  \cite{Barn:Blae:1983}.

%\begin{definition}\label{def2}   Let $M$ be the set  of discrete strictly positive measures on a   ${\cal P} \subseteq \mathcal B$. Define an equivalence relation $\sim$ on $M$ by $x \sim x^\prime$ if and only if there exists $\lambda >0 $ such that $x = \lambda x^\prime$. The quotient space is defined as  $\xe:= M/\sim$  \cite{Murr:Rice:1993}.   Let $V_{exp}$ be the vector space of real valued functions over ${\cal P}$ , then the   $+1$-affine space  for distributions with support ${\cal P}$ is then given by
%$(\xe, \ve, \oplus)$ where $$<  x> \oplus v = <xe^v>,$$ where $< >$ denotes equivalence class.
%\end{definition}
%

The following definition gives a useful computational tool for understanding  the limiting behaviour of exponential subfamilies in $\Delta^k$, and gives a generalisation of the trinomial model   shown in Fig.~\ref{dualityplot}.

\begin{definition}\label{simplicialparameterdefinition}
Let $\pi^0=\left( \pi^0_0, \dots, \pi^0_k  \right)$ be a probability vector, $a_1, \dots, a_d$ be a set of vectors in ${\mathbb R}^{k+1}$, 
such that
$$
1_{k+1}, a_1, \dots, a_d 
$$  are linearly independent, and $b_1, \dots, b_{k-d}$ be  a set of linearly independent vectors in $\vm$ such that $a_i^T b_j=0$ for $i=1, \dots, d$ and $j=1, \dots, k-d$.  Furthermore, define 
$$\bar{ P}_{\pi^0}  := \left\{ (\lambda, \sigma) : \left(p_{\pi^0}(\lambda, \sigma)\right)_h \ge 0 {\rm \;for \;all\;} h=0, \dots, k \right\},$$ in which  $\lambda \in {\mathbb R}^d$, $\sigma \in {\mathbb R}^{k-d}$ and
\begin{equation}\label{simplicialparameterfunction}
 \left(p_{\pi^0}(\lambda, \sigma)\right)_h :=   \frac{ \left(\pi_h^0+ \sum_{j=1}^{k-d} (\sigma_j b_j)_h\right) \exp\{ \sum_{i=1}^d (\lambda_ia_i)_h\}    }{ \sum_{h^*=0}^k \left\{ \left(\pi_{h^*}^0+ \sum_{j=1}^{k-d} (\sigma_j b_j)_{h^*}  \right) \exp\{ \sum_{i=1}^d (\lambda_ia_i)_{h^*}\}  \right\}  },  \end{equation} where 
\begin{equation}\label{positivity of sigma}
\left(\pi_h^0+ \sum_{j=1}^{k-d} (\sigma_j b_j)_h\right) \ge 0.
\end{equation}

\end{definition}

Note that for fixed $\sigma=\sigma^0$ the image of $p_{\pi^0}(\cdot, \sigma^0)$ is a $d$-dimensional exponential family. As $\sigma^0$ changes  these exponential families are $+1$-parallel. However for fixed $\lambda=\lambda^0$, the image of $p_{\pi^0}(\lambda^0, \cdot)$ is not in general $-1$-affine, but is for the special case when $\lambda^0=0$. Thus this construction, while having the advantage of being explicit, is not as strong as a true mixed parameterisation. 
However, the function defined in Definition \ref{simplicialparameterfunction} is  a useful tool in understanding the limiting properties of exponential families within the extended multinomial model. Consider    the set of possible values of $\sigma$. By condition ($\ref{positivity of sigma}$) it follows  that the domain of $\sigma$ -- for given $\pi^0$ -- is  a polytope. As  $\sigma$ converges to the boundary of this polytope the corresponding exponential family converges to an extended exponential family defined on the boundary of $\Delta^k$ determined by the corresponding zeros in the probability vector. This construction generalises the plots in Fig.~\ref
{dualityplot} (c) and (d). Notice also that  it allows the definition of the limits of families which complements the pointwise limits defined in  \cite{Brow:1986} and \cite{Csis:Matu:2005}.

%
%  A similar structure is obtained using the mixed parameterization (see \cite{Brow:1986}, \S 3.10).
%The simplicial parameterization defined in Def.~\ref{simplicialparameterdefinition} is explicit, which is a property which will prove illuminating in what follows. Furthermore, the simplicial parameterization describes  the whole simplex including the boundary.  Locally to a given $\pi^0$ in the interior of the simplex, the simplicial parameterization is equivalent  to the mixed parameterization. Although the mixed parameterization does not require to make reference to a point in the simplex the advantage of the simplicial over the mixed parameterization is that it can be calculated explicitly.

\subsection{Computing limits in exponential families}\label{Computing limits in exponential families}

 \begin{figure}[htbp] %  figure placement: here, top, bottom, or page
   \includegraphics[width=2in]{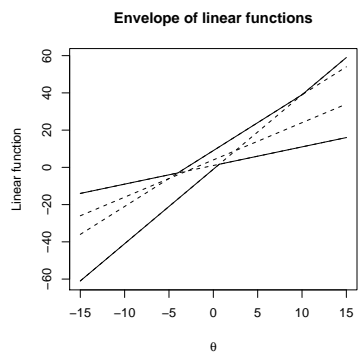}
    \caption{The envelope of a set of linear functions. Functions: dashed lines, envelope: solid lines}
    \label{simplexenvelope.eps}
 \end{figure}
\begin{example}\label{Simple limit example}
In order to visualise the geometric s of the problem of computing limits in exponential families consider a  low dimensional example.  Define   a two dimensional full exponential  family  by the  vectors $v_1= (1,2,3,4), v_2=(1,4,9,-1)$ and the uniform distribution base point, embedded in the three dimensional simplex.  The $2$-dimensional family is defined by the $+1$-affine space  through $(0.25, 0.25,0.25,0.25)$ spanned by the space of vectors of the form
$$\alpha (1,2,3,4) + \beta (1,4,9,-1)= (\alpha+\beta,  2\alpha + 4\beta, 3\alpha+ 9\beta, 4\alpha -\beta)$$
Consider directions from the origin  found by writing  $\alpha=\theta \beta$ giving, for each $\theta$,  a   one dimensional full exponential family parameterized by $\beta$ in the direction 
$\beta(\theta +1, 2\theta+4, 3\theta+9, 4\theta- 1  )$.  The aspect of this vector which determines the connection  to the boundary  is the rank structure of its elements. For example, suppose the first component was the maximum and  the last the minimum, then as $\beta \rightarrow \pm \infty$ this one dimensional family will be connected to the first and fourth vertex of the embedding four simplex, respectively. Note that     changing the value of $\theta$  changes the rank structure, as illustrated in Fig.~\ref{simplexenvelope.eps}. In this plot,  the four linear functions of $\theta$ are plotted (dashed lines) and the the impact of rank structure is determined by the upper and lower envelopes (solid lines). From this analysis of the envelopes of a set of linear functions it can be seen that the function $2\theta +4$ is redundant. The consequence of this is shown in Fig.~\ref {boundaryexample.eps} which shows the result of  direct computation in the two dimensional family. It is clear that, indeed, only three of the four vertexes of the ambient 4-simplex have been connected by the model.
\end{example}

In general,  the problem of finding the limit points in full exponential families inside simplex models is a problem of finding redundant linear constraints. As shown in 
\cite{Edels:Algorthms:1987},  this can be converted, via duality, into the problem of finding extremal points in a finite dimensional affine space. 
%
% In the four cycle model this technique can construct all possible boundary  sub-simplices   which lie on the four cycle model. For example, it can be shown  that  all of the 16 vertices are part of the boundary.  This analysis then extends to the distribution of maximum likelihood estimate. Once the boundary points have been identified necessary and sufficient conditions for the existence of the maximum likelihood in the $+1$-parameters can easily be found computationally. 
  \begin{figure}[htbp] %  figure placement: here, top, bottom, or page

   \includegraphics[width=2in]{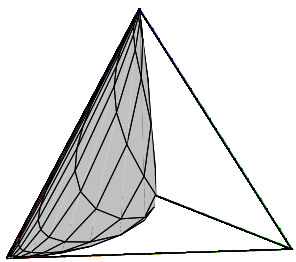}
    \caption{Attaching a two dimensional example to the boundary of the simplex}
    \label{boundaryexample.eps}
 \end{figure}

 \setcounter{foo}{\value{example}}

\setcounter {example} {2}
\begin{example}[continued]
Consider an $N \times D$ design matrix $X$ with $N$ samples and a binary response $t \in \{0, 1\}^N$. Let 
$s(x) = \log\left( \frac{x}{ 1- x} \right)$ so that  $s^{-1}(x) = \frac{\exp(x)}{1+\exp(x)}$, the  logistic regression  model being  given by 
$$P(T_i=1)= s^{-1}(\beta^T X^T_{i, })$$ where $X_{i, }$ is the $i^{th}$ row of $X$.
This is a full exponential family that   lies in the $2^N-1$ simplex when considered a model for the joint distribution of the $N$ binary response variates. A design matrix  $X$ defines  a $D$-dimensional $+1$-affine subset and changing the explanatory variates  changes the direction of this low dimensional space inside the space of joint distributions.

Consider response data (0, 1, 0, 1, 0, 1, 1), the explanatory variables being $x_0= (1,1,1,1,1,1,1)$ and $x_1=(1,2,3,4,5,6,7)$.  For convenience, in the space of all joint distributions,  label the bin associated with the sequence $\{t_i\}_{i=0}^{N-1}$ with the binary number which that sequence represents
\begin{equation}\label{binary representation}
\sum_{i=0}^{N-1} 2^i t_i.
\end{equation} This  logistic model is a two-dimensional exponential family which passes through the  point  corresponding to the uniform distribution of the $2^N$ simplex and  lies in the directions defined by 
$$v_0:= \left(\sum_{i=0}^{N-1}  t_{ij} x_{0 i}\right)_{j=1}^{2^7} , v_1:= \left( \sum_{i=0}^{N-1}  t_{ij} x_{1 i}\right)_{j=1}^{2^7} ,$$ where $t_{ij}$ is the binary representation of vertex $j$.  
\begin{figure}[htbp] %  figure placement: here, top, bottom, or page
   \centering
   \includegraphics[width=2in]{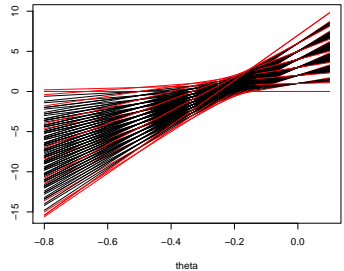} 
   \caption{Envelopes of lines  }
   \label{logisticfigenvelope.pdf}
\end{figure}

As in Example \ref{Simple limit example}  consider the way that this  two-dimensional exponential family is attached to the boundary using the envelope method.  There are $2^n$ possible lines to consider and these are shown in Fig.~\ref{logisticfigenvelope.pdf}. These lines whose duals are extremal points are plotted in red and it can clearly be seen that the upper and lower envelopes have been found.  The corresponding vertices which the full exponential family reaches are given by  vectors of the form $z$ with the structure either  $z_i=0$ $i=1, \dots, h$ and $1$ for $i=h+1, \dots, N$ or $z_i=1$ $i=1, \dots, h$ and $0$ for $i=h+1, \dots, N$.

%
%$$(z_1, \dots, z_n,  $$
%$$
%\begin{array}{ccccccc}
%  0 &    0 &   0 &   0   & 0   & 0 &   0\\
%    0 &   0   & 0    &0  &  0   & 0   & 1 \\
%    0 &   0  &  0   & 0  &  0 &   1  &  1\\
%   0  &  0   & 0   & 0   & 1  &  1 &   1\\
%    0  &  0  &  0  &  1  &  1 &   1 &   1\\
%  0   & 0   & 1   & 1   & 1   & 1   & 1\\
%   0  &  1 &   1 &   1  &  1  &  1 &   1\\
%    1 &   1 &   1  &  1 &   1 &   1 &   1\\
%   1   & 0   & 0   & 0   & 0  &  0  &  0\\
%    1  &  1 &   0 &   0 &   0 &   0&    0\\
%  1   & 1   & 1   & 0   & 0   & 0  &  0\\
%  1  &  1  &  1  &  1  &  0  &  0 &   0\\
%  1   & 1  &  1  &  1  &  1  &  0  &  0\\
%   1  &  1  &  1  &  1  &  1 &   1  &  0\\
%    \end{array}
%$$
We can see how this global geometry affects the inference. One immediate issue is that if the observed data is a sequence which is one of the vertices listed above then the corresponding MLE will also lie on the boundary. Thus, for example, if the  observed data is $(0,1, 0,1,0, 1,1)$ there is a `regular' turning point in $\beta$-space. However if, instead,  the data is  $( 1, 1, 0, 0, 0, 0 ,0)$ the  MLE does indeed go to infinity and has its maximum at the correct vertex.
This result for $N=7$ in fact generalizes, when the explanatory variable is linear, to any $N$. The corresponding vertices which the full exponential family reaches are  again given by vectors of the form $z$ with one of the two structures identified above.
%\begin{figure}[htbp] %  figure placement: here, top, bottom, or page
%   \centering
%   \includegraphics[width=1.4in]{art/logisticfiginternal.eps} 
%      \includegraphics[width=1.4in]{art/logsiticminboundary.eps} 
%   \caption{The behaviour of the likelihood function}
%   \label{logisticfiginternal.pdf}
%\end{figure}

%To illustrate this look at Fig.~\ref{logisticfiginternal.pdf} The left hand panel shows the llkelihood contours in black for the observed data $(0,1, 0,1,0, 1,1)$. There is a `regular' turning point in $\beta$-space. The red lines divide up the space so that $+1$-rays from $(0,0)$ in  a given segment connect to a given vertex.  If we change the data to say $( 1, 1, 0, 0, 0, 0 ,0)$ (right hand panel) we see the MLE does indeed go to infinity and has its maximum on the correct vertex. {\bf REF F and R (2012?)}
%
%

\end{example}

\setcounter{example}{\value{foo}}

\section{The tools of information geometry}\label{Classical information geometry}

 In general, working in a simplex, boundary effects mean that standard first order asymptotic results can fail. Most standard methods are not uniform across the simplex. Therefore one way that the higher order asymptotic methods of information geometry have value is that they can be  used to validate the region of parameter space where the first order method will be accurate. Example   \ref{censored exponential  example}  has a continuous random variable with compact support and it is used to show how discretisation can be  used to  apply computational information geometry to such models.

\subsection{Higher order asymptotics: Edgeworth expansions}\label{Higher order asymptotics: Edgeworth expansions}

 One very powerful set of results from classical information geometry derives from the fact that geometrically based tensor analysis is well-suited for use in multi-dimensional higher order asymptotic analysis, see \cite{Barn:Cox:1989} or \cite{McCu:1987}. However, using this tensorial formulation is not without difficulty for the mainstream statistician. Its very efficient, tight notation may perhaps obscure rather than enlighten, while the resulting formulae can typically have a very large number of terms, making them rather cumbersome to work with explicitly. These obstacles to implementation are overcome by the computational approach described in this paper. The clarity of the tensorial approach is ideal for coding, while  large numbers of additive terms, of course, are easy to deal with.  Two more fundamental  issues, which the global geometric approach of this paper highlight,  concern numerical stability.  The ability to invert  the Fisher information matrix is vital in most tensorial formulae and so understanding its  spectrum, as discussed in Section \ref{Spectrum of Fisher Information}, is vital. Secondly numerical under and overflow near boundaries  requires careful analysis and so understanding the way that models are attached to the boundaries of the extended multinomial models is equally important. 
 
An important aspect  of  higher order methods is not just their accuracy in a given example, but the way that they can be used to validate first order methods. In cases like logistic regression  first order methods are typically used for inference despite the fact that they are not uniformly accurate across the parameter space of interest. In the example below the fact that the Edgeworth expansion is far from normal acts as a diagnostic for the first order methods.

\setcounter{foo}{\value{example}} 

\setcounter {example} {2}
\begin{example}[continued]
Consider Fig.~\ref{edgeworthtwodim} where the parameters of a two dimensional logistic
family are such that the sampling distribution of the sufficient
statistics is considerably far from normal. This is shown by the simulated
sample of black points, the red contours, computed numerically from the Edgeworth expansion,  showing a good fit with the simulation,
but a distribution which is far from the first order approximation. As holds widely, in this example, the Edgeworth expansion is easy to compute numerically.

%
%The multivariate Edgeworth approximation to the sampling distribution of part of the sufficient statistic for the four cycle model  is shown in  Fig.~\ref{edgeworthtwodim}.  Using the techniques described in \S \ref{Computing limits in exponential families} a point near the boundary of the $15$-simplex has been selected as the data generation process. For illustration the marginal distribution of two components of the  sufficient statistic of Example \ref{undirected graphical models} is selected, though any dimension could have be chosen.  The boundary forces constraints on the range of the sufficient statistics which is shown by the dashed  line in the plot. The points, jittered for clarity, show the distribution computed by simulation. It is typical that  such boundary  constraints stop standard first order methods from working well but the greater  flexibility of higher order methods can be seen to work well here. As discussed methods such as the multivariate Edgeworth expansion can be strongly exploited in a computational framework.
%

\begin{figure}[h] %  figure placement: here, top, bottom, or page
   \centering
   
   \includegraphics[width=3in]{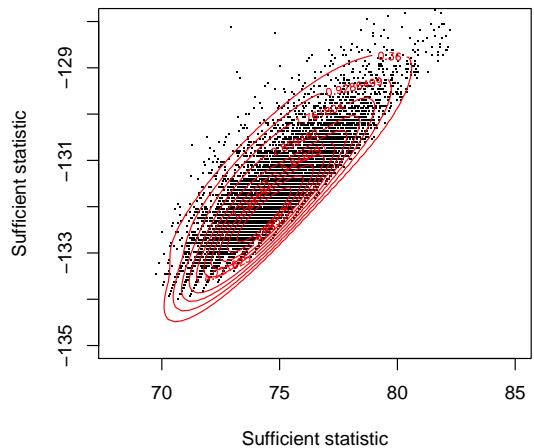}
   \caption{Using the Edgeworth expansion near the boundary of Example 3}
   \label{edgeworthtwodim}
\end{figure}

\end{example}
\setcounter{example}{\value{foo}}

\subsection{Continuity and compactness}\label{Continuous and compact}

In order to use the high dimensional simplex models with continuous random variables it  is necessary  to truncate and discretise the sample space into a finite number of bins. The following theorems show that the information loss in doing this is arbitrarily  small for a fine enough discretisation and that the key to understanding the information in general is controlling the conditional moments in each bin  of the random variables of interest, uniformly in the parameters of the model.

\begin{thm}\label{general information loss}
Let $f(x;\theta)$, $ \theta \in \Theta$, be a parametric family of density functions  with  common support ${\cal  X} \subset {\bf R}^d$ each being continuously differentiable on the relative interior of ${\cal  X}$, assumed non-empty.  Further, let ${\cal X}$ be compact,while  $$\left\{ \left\| \frac{\partial}{\partial x} f(x;\theta) \right\| | x \in {\cal X} \right\}$$ is uniformly bounded in $\theta \in \Theta$ by $M$, say.

Then for any $\epsilon >0$ and for any sample size $N>0$, there exists a finite, measurable partition $\left\{ B_k \right\}_{k=0}^{K(\epsilon, N)}$ of ${\cal X}$ such that:  for all $(x_1, \dots, x_N) \in {\cal X}^N$, and for all $(\theta_0, \theta) \in \Theta^2$
\begin{equation}\label{general likelihood bound}
\left| \log\left\{ \frac{ Lik_d(\theta) }{  Lik_d(\theta_0) }  \right\}  -   \log\left\{  \frac{ Lik_c(\theta) }{  Lik_c(\theta_0) }   \  \right\}      \right| \le \epsilon,
\end{equation}  where $Lik_d$ and $Lik_c$ are the likelihood functions from the discretised and continuous distributions respectively.
\end{thm}
\begin{proof}  See Appendix.
\end{proof}

The following result looks at the case where the family that is discretised is itself an exponential family  and so the tools of classical information geometry can be applied. In general, after discretisation a full exponential  family does not remain full exponential and there is  information loss. However, the following results show that this loss can be made small enough to be unimportant for inference and that all information geometric results on the two families can be made  arbitrarily close. 
\begin{thm}
\label{exponential information loss theorem} 
Let $f(x;\theta)= \nu(x) \exp\left\{\theta^T s(x)- \psi(\theta)  \right\}$, $x \in {\cal X}, \theta \in \Theta$,  be 
an exponential family which satisfies the regularity conditions of \cite{Amar:1990}, p.~16. Further, assume that $s(x)$ is uniformly continuous and $s({\cal X})$ is compact. 

Then, for any $\epsilon > 0$, there exists a finite measurable partition $\{B_k\}_{k=0}^{K(\epsilon)}$ of ${\cal X}$ such that, for all choices of bin labels $s_k \in
s(B_k)$, all terms of Amari's information geometry  for $f(x; \theta)$ can be approximated to $O(\epsilon)$  by the corresponding terms for
the family
$$
\left\{ (\pi_k(\theta), s_k) |   \pi_k(\theta) = \int_{B_k} f(x;\theta) dx, s_k \in s(B_k) \right\}.
$$ In particular:
\begin{enumerate}
\item[(a)] For all $\theta$, and any norm,
$$\| \mu_d(\theta) - \mu_c(\theta)  \|= O(\epsilon)$$where $\mu_d(\theta)= \sum_{k=0}^{K(\epsilon)} s_k \pi_k(\theta)$  and $\mu_c(\theta)=  \int_{\cal X} x f(x;\theta) dx$.

\item[(b)] The expected Fisher information for $\theta$ of $f(x;\theta)$, $I_c(\theta)$, and the  expected Fisher information for $\{\pi_k(\theta)\}$, $I_d(\theta)$, satisfies 
$$\| I_d(\theta) - I_c(\theta)  \|_\infty  =O(\epsilon^2).$$

\item[(c)] The skewness tensors $T_c(\theta)$, see \cite{Amar:1990}, p.~105,  of $f(x;\theta)$  and $T_d(\theta)$ for $\{\pi_k(\theta)\}$ satisfy 
$$\| T_d(\theta) - T_c(\theta)  \|_\infty  =O(\epsilon^3).$$
\end{enumerate}

\end{thm}

\begin{proof}  See Appendix.
\end{proof}

The following Corollary states that the likelihood before and after discretisation  can also be made arbitrarily close with a fine enough discretisation, as illustrated in Fig.~\ref{example2fig2.pdf} drawn from Example 2, as described below.

\begin{corollary}\label{corollary for exponential family} Under the conditions of Theorem \ref{exponential information loss theorem}, let $\hat \theta_c$ denote the MLE based  on a sample, $x_1, \dots, x_N$,  from $f(x, \theta)$ and $\hat \theta_d$ the MLE for $\{\pi_k(\theta)\}$ based on the counts $n_k$, $k=0,  \dots, K(\epsilon) $ 
for the partition $\{B_k\}_{k=0}^{K(\epsilon)}$ of Theorem  \ref{exponential information loss theorem}.

Then \begin{equation}\label{difference in MLE} 
\| \hat \theta_d - \hat \theta_c \| = O(\epsilon)
\end{equation} 
and
 \begin{equation}\label{difference in information } 
\left|  \frac{ \partial^2 \ell_d }{ \partial \theta_r \partial \theta_s  }(\hat \theta_d)   -   \frac{ \partial^2 \ell_c }{ \partial \theta_r \partial \theta_s  }(\hat \theta_c)  \right| = O(\epsilon)
\end{equation} 

\end{corollary}

\begin{proof} See Appendix.\end{proof}

The following example illustrates these results and also shows an application  of  dimension reduction based on information geometry.  Dimension reduction  is dependent on the choice of affine structure. The reduction here is done in the $+1$-affine geometry,  unlike the mixture geometry examples, \ref{binomial example}  and \ref{tripod example},  where it is done in the $-1$-geometry.

\setcounter{foo}{\value{example}}

\setcounter {example} {2}
\addtocounter {example} {-1}
\begin{example}[continued]
This example shows  how results from information geometry can be  numerically implemented in the 
resultant curved exponential family.  An example   in \cite{Hand:Daly:Lunn:1994} concerns survival times $Z$ for  leukaemia patients measured in days from the time of diagnosis. Originally from \cite{Brys:Sidd:1969}, there are  43 observations. For  illustrative purposes the data is censored at a fixed value such that the censored exponential distribution  gives a reasonable, but not perfect, fit.  It is  assumed  the random variable $Z$ has an exponential distribution but only  $Y= \min\{ Z, t\}$ is observed.  As discussed in \cite{Marr:West:On:2002} this gives a one-dimensional curved exponential family inside a two dimensional regular exponential family of the form 
\begin{equation}\label{eq1}
\exp\left[\lambda^1 x +\lambda^2 y -
\log\left\{
\frac{1}{\lambda^2}\left(e^{\lambda^2 t}
-1 \right)+ e^{\lambda^1 +\lambda^2 t} \right\} \right] ,
\end{equation}
where $y=\min(z, t)$ and $x=I(z\ge t)$ and the embedding map is given by 
$(\lambda^1(\theta), \lambda^2(\theta))=(-\log \theta, -\theta)$.

  \begin{figure}[htbp] %  figure placement: here, top, bottom, or page
    \centering
   \includegraphics[width=3in]{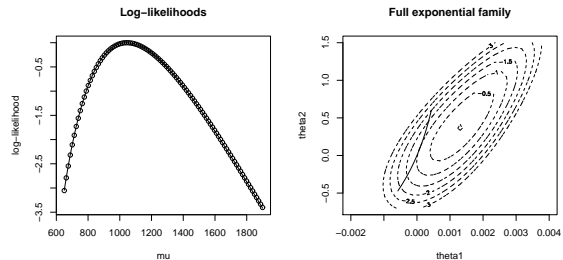}
    \caption{Computational information geometry: likelihood approximation and dimension reduction}
    \label{example2fig2.pdf}
 \end{figure}

   \begin{figure}[htbp] %  figure placement: here, top, bottom, or page
    \centering
   \includegraphics[width=3in]{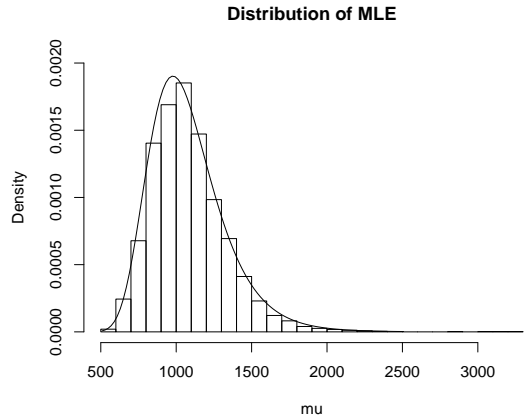}
    \caption{Sampling distribution of $\widehat \mu$ for censored exponential based on saddlepoint approximation }
    \label{example2fig4.pdf}
\end{figure}
 
 Figure \ref{example2fig2.pdf} shows some of the details of the geometry of the curved exponential family which is created after censoring. The censoring value was chosen at 750. The parameter of interest is $\mu$, the mean of the uncensored observations.  In the left hand panel of Fig.~\ref{example2fig2.pdf}, the solid line is the likelihood function based on binning the data to bins of width four days and using a multinomial approximation.  The dots in this panel are the log-likelihood for the raw data based on the continuous censored exponential model. As can be clearly seen there is no real inferential loss in the binning and discretisation process.  The likelihood plot also shows appreciable skewness,  which suggests that standard first order asymptotics might be  improved by the higher order  asymptotic methods of classical information geometry.

 The right hand panel shows the censored exponential (solid curve) embedded in the two-dimensional full exponential family in the $+1$-parameterization. The dashed contours are the log-likelihood contours in the full exponential family. It is clear, even  visually, that there is not much $+1$ curvature  for this family on this inferential scale. So this is an example where  the curved exponential family behaves inferentially like a one-dimensional full exponential family. In particular, the dimension reduction techniques found in \cite{Marr:Vos:On:2004},  can be used.  To see the effectiveness of this  idea,  Fig.~\ref{example2fig4.pdf} shows how well a saddlepoint based approximation does at approximating the distribution of the maximum likelihood estimator of the parameter of interest. \end{example}
\setcounter{example}{\value{foo}}

\subsection{Higher order asymptotics: saddlepoint method}\label{Computational Consequences}

 The saddlepoint approximation method is very important tool from classical information geometry, see Fig.~\ref{example2fig4.pdf} for an example. Using this method requires the solving of the so-called saddlepoint equation in an efficient and accurate manner and so for computational information geometry this only  needs to be done numerically. The problem of solving this non-linear equation is  tied to understanding the non-linear relationship between the $+1$ and $-1$-parameters, and hence the rigorous implementation of numerical methods requires understanding  the global geometry described above. For example, the issues surrounding such implementation being far from uniform across the simplex, it will help to be made aware if the method is being attempted in a region where first order asymptotics would work well or not.

\setcounter{foo}{\value{example}}
\setcounter {example} {1}
\begin{example}[continued]
Example \ref{censored exponential  example} is a  curved exponential family, \cite{Kass:Vos:1997}.  Consider Fig.~\ref{example2parafig.eps}, this shows the level sets of the mean parameterization for the $2$-dimensional family  plotted in the natural parameters. Solving the saddlepoint equation requires mapping between these two coordinate systems.  The figure illustrates the issues which need considering in implementing numerical methods to do this. 
At point `A' in the figure we see that the level sets of the mean parameter are becoming close to parallel -- this reflects the fact that the Fisher information can be very close to singular, as discussed in \S \ref{Spectrum of Fisher Information}. At the point `B' the bifurcation in the parameters, described in \S \ref{Finite dimensional case}, is clear. Again,  the point `C' shows a region where there is close to linearity between the two coordinate systems which is typical of when first order asymptotic methods work well, see \S \ref{Higher order asymptotics: Edgeworth expansions}.

  \begin{figure}[h] %  figure placement: here, top, bottom, or page
    \centering
   \includegraphics[width=2in]{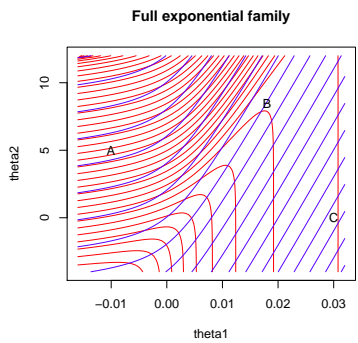}
    \caption{Mean parameterization (blue and red lines) plotted in the natural parameters}
    \label{example2parafig.eps}
 \end{figure}

\end{example}

\setcounter{example}{\value{foo}}

\section{Inference on Mixtures}\label{Mixture Inference}
\subsection{Lindsay's geometry and the simplex}

This section describes the way the mixture geometry of \cite{Lind:mixt:1995}  is related to  the information geometry of the  simplex. In particular, it will lead to extending Lindsay's structure in a way which will give considerable  computational advantages in, for example, computing the non-parametric maximum likelihood estimate of a mixture model and understanding its variability.

 \begin{figure}[htbp] %  figure placement: here, top, bottom, or page
   \centering
   \includegraphics[width=3in]{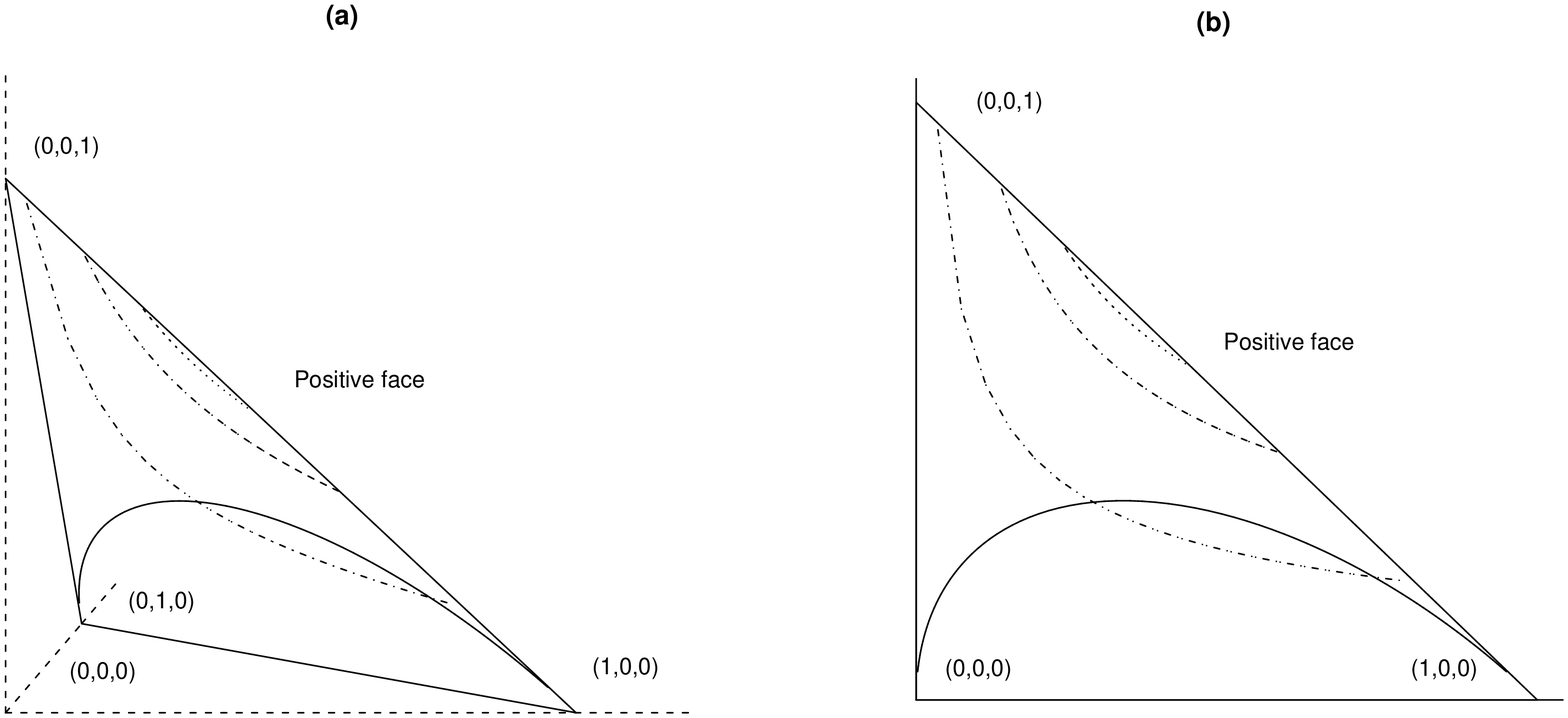}
   \caption{(a) The simplex with a one-dimensional full exponential family (solid) and likelihood contours (dashed) (b) The image of the simplex under the map $\Pi_L$ }
   \label{lindsayfig1}
\end{figure}

Lindsay's geometry  lies in an affine space which is determined by the observed data. In particular, it is always finite dimensional, and  the dimension is determined by the number of distinct observations.   Following the notation of \cite{Lesp:Kalb:1992}, which  looks at mixtures of the model $h(y |\theta)$  i.e. models of the form $f(y ; Q)= \int h(y |\theta) dQ(\theta)$, let 
$L_\theta= (L_1(\theta), \dots, L_{N^*}(\theta))$ represent the $N^*$ distinct likelihood values of $h(y_i |\theta)$ arising from the data $y_1, \dots y_n$. 
The likelihood on the space of mixtures is defined on the convex hull of the image of the map
$$\theta \rightarrow (L_1(\theta), \dots, L_{N^*}(\theta)) \subset {\rm R}^{N^*}.$$  Then the problem of finding the non-parametric likelihood estimate, determined by $\widehat Q$, is found by maximising a concave function over this convex set.

There are clear parallels between the convex geometry of Lindsay and the embedding in the $-1$-simplex. Lindsay's geometry is designed for working with the  likelihood so only concerns the observed data, rather than the full sample space. For simplicity consider discrete models where the distinct likelihood components are represented by probabilities $\pi_i$ where, by definition, $i$  lies  in the observed face  $\mathcal{P}$ defined in Theorem  \ref{shape of likelihood}  (Section 2.3).  The affine structure of Lindsay is  thus determined by the vertices of   $\mathcal{P}$,  see Fig.~\ref{lindsayfig1}. 
 
 \begin{definition}\label{Projection to lindsay}
  Define $\Pi_L$ to be the Euclidean orthogonal projection from a simplex to the smallest vector  space containing the vertices indexed by ${\mathcal P}$.
 \end{definition}

The following result is strongly connected to Theorem \ref{shape of likelihood}. In it, the level sets of the likelihood are now characterised as the pre-images of the mapping $\Pi_L$. It also shows that searching for the maximum likelihood in the convex hull in the simplex is the same as in Lindsay's geometry.

 \begin{thm}\label{lindsay embedding theorem}
 a) The likelihood on the simplex is completely determined by the likelihood on the image of $\Pi_L$.  In particular, all elements of the pre-image of  $\Pi_L$     have the same likelihood value. \\ (b)   $\Pi_L$  maps   $-1$ convex hulls  in the  $-1$-simplex to the  convex hull of Lindsay's geometry. 
 \end{thm}
\begin{proof}
See Appendix. 
\end{proof}
%\begin{center}
%   \includegraphics[width=3in]{Talk/cigtalkfig4.pdf} 
%\end{center}

Given  this result, it is natural to study the likelihood of a convex hull in the simplex rather than in Lindsay's space. There are some definite advantages to this, some of which will be explored in this paper, while others will only be briefly mentioned.   In Sections  \ref{Total positivity and local mixing} and \ref{Implementation of Algorithm} a new search algorithm is proposed which exploits the information geometry of the simplex. In particular, it exploits  dimension reduction directly in  the simplex to give a direct way of computing the non-parametric maximum likelihood estimate. 

A further advantage of working in the simplex is that while Theorem \ref{lindsay embedding theorem} shows that Lindsay's geometry captures the $-1$ and likelihood structure, it does not capture the full information geometry. For example, the expected  Fisher information cannot be represented, since it is a defined using the full sample space, and hence analysis of the variability of the non-parametric maximum likelihood estimate  is more natural in the full simplex, rather than in the data-dependent space proposed by Lindsay.

\subsection{Total positivity and local mixing}\label{Total positivity and local mixing}

In order to consider dimension reduction in the $-1$ simplex, and the corresponding dimension of the convex hull, this paper concentrates on the case where the mixture is over an exponential family. At first sight, Theorem \ref{total positivity} and the following comments may appear contradictory.   First  Theorem \ref{total positivity} shows  that $-1$-convex hulls of full exponential families have maximal dimension in the simplex,  whereas the concept of local mixing, and its extension to polytope approximation in  Theorem \ref{likelihood approximation  theorem},  shows that there exist very good low dimensional approximations to these convex hulls.  It is the existence of these low dimensional approximations which is exploited by the proposed algorithm. Using results on  total positivity, we have

\begin{thm}\label{total positivity}
The $-1$-convex hull of an open subset of a generic one dimensional exponential family is of full dimension.
\end{thm}

\begin{proof}
See Appendix.
 \end{proof}
In  this result ``generic'' means that the $+1$ tangent vector which defines the exponential family has components which are all distinct.

Theorem \ref{total positivity} can be contrasted with the results of  \cite{Marr:on:2002}  or \cite{Anaya:Marriott:local:2007} which state, under regularity and  for many applications,  mixtures of exponential families have accurate low dimensional representations.  The essential resolution of this apparent contradiction is that if the segment of the curve $\pi(\theta)$ for $\theta \in \Theta$ lies `close' to a low dimensional $-1$-affine subspace, then all mixtures over $\Theta$ also lie `close' to this space. The following discussion is then concerned with the appropriate definition of `close' for modelling purposes.

Motivated by the idea of a local mixture,  consider how well a full exponential family $\pi(\theta)$ can be approximated by a $-1$ polygonal path which vertices $\pi(\theta_i)$,  $i=1, \dots, M$.  Any point on this polygonal path will have the form 
\begin{equation}\label{definition of polygon}
\rho \pi(\theta_i) + (1-\rho) \pi(\theta_{i+1})
\end{equation} with $\rho \in [0,1]$. Define the segment $S_i:= \left\{ \rho \pi(\theta_i) + (1-\rho) \pi(\theta_{i+1})  | \rho \in [0,1] \right\}$. So,  on top of  the usual label switching identification issue with mixtures,  there is additionally the identification problem induced by 
\begin{equation}\label{segment Identifcation}
\int \left\{   \rho \pi(\theta_i) + (1-\rho) \pi(\theta_{i+1})\right\} dQ(\rho)  =  \int \left\{   \rho \pi(\theta_i) + (1-\rho) \pi(\theta_{i+1})\right\} dQ^\prime (\rho)  
\end{equation}  when $E_Q(\rho)= E_{Q^\prime} (\rho)$.  While lack of  identification is usually considered a statistical problem,  computationally  it restricts  the space  the likelihood needs to be optimised over.  It will be shown that  restricting attention to this  space has considerable computational advantages. 

Consider, then, the following definition and lemma.
\begin{definition}\label{definition of distance function} Given a norm $\| \cdot \|$, 
the curve $\pi(\theta)$ and the polygonal path $\cup S_i$ define the distance function by
$$d(\pi(\theta)) := \inf_{\pi \in \cup S_i}  \left\| \pi(\theta) - \pi  \right\|.$$
\end{definition}

\begin{lemma}\label{lemma on uniform distance}
If $d(\pi(\theta)) \le \epsilon$  for all $\theta$ then 
any point in the convex hull of $\pi(\theta)$ lies within $\epsilon$ of the convex hull of the finite set
$\pi(\theta_i)$.
\end{lemma}

\begin{proof} By the triangle inequality. \end{proof}

Let $ \hat \pi ^{NP} $ be the non-parametric maximum likelihood estimate for mixtures of the curve $\pi(\theta)$.  A  consequence of Lemma \ref{lemma on uniform distance} is that, under the uniform approximation assumption, $ \hat \pi ^{NP} $  lies within $\epsilon$ of the convex hull of the polygon.   The question is then what norm is appropriate for measuring the quality of the polygonal approximation. 

\begin{definition}\label{definition of inner production}
Define the inner product 
$$\left< v, w \right>_\pi:= \sum_{i=0}^k \frac{v_i w_i}{\pi_i}$$ for $v, w \in \vm$ and $\pi$ such that $\pi_i > 0$ for all $i$. This defines a preferred point metric as discussed in \cite{Crit:Marr:Salm:pref:1993}. Further,  let $\|\cdot\|_\pi$ be the corresponding norm.
\end{definition}
As motivation for using such a metric, consider the Taylor expansion  for the likelihood
around $\hat \pi$ when the maximum is defined by turning point conditions, i.e. occurs at  a point in the interior of the simplex. Under these conditions, to high order, it follows that 
\begin{equation}\label{likelihood expansion turning point}
\ell(\pi) - \ell(\hat \pi) \approx -\frac{N}{2}  \left\| \pi- \hat \pi  \right\|_{\hat \pi}^2. 
\end{equation}
So small dispersions, as measured by $\|\cdot\|_{\hat \pi}$,  correspond to  small changes in likelihood values. Note that this  is  clearly not true under the standard Euclidean norm, where unbounded changes in likelihood values are possible.

Following  \cite{Lind:mixt:1995},   the maximum of the likelihood in a convex hull is determined by  the non-positivity of directional derivatives, rather than turning points.  So the following likelihood approximation theorem is appropriate.

\begin{thm}\label{likelihood approximation  theorem}
Let $\pi(\theta)$ be an exponential family, and $\{\theta_i\}$ a finite  and fixed set of support points  such that  $d(\pi(\theta)) \le \epsilon$ for all $\theta$. Further, denote by  $\hat \pi^{NP}$ and $\hat \pi$ the maximum likelihood estimates in the convex hulls of $\pi(\theta)$ and $\left\{ \pi(\theta_i) | i =1, \dots, M\right\}$ respectively,  and by $\hat \pi_i^G:= \frac{n_i}{N}$ the global maximiser in the simplex. Then, 
\begin{equation}\label{likelihood bound} 
\ell(\hat \pi^{NP}) - \ell (\hat \pi) \le  \epsilon N  || (\hat \pi^{G}-  \hat \pi^{NP} )     ||_{\hat \pi}+ o(\epsilon)
\end{equation}

\end{thm}

\begin{proof} 
See Appendix.
\end{proof}
\subsection{Implementation of Algorithm}\label{Implementation of Algorithm}

 Algorithms using the polygonal approximation technique will be evaluated in detail in future work. Here a general outline is given and 
 a couple of examples examined (Examples 1 and 4).  The fundamental idea is to compute the convex hull of a finite number of points on the curve as an approximation to the convex hull of the curve itself.  The positioning of the points can be decided by using singular value decomposition methods to see if the $+1$ line segment  joining  consecutive points has small enough $-1$ curvature. From these it is necessary to compute $\epsilon$ which bounds the  uniform approximation of the curve by the polygon and then apply Theorem \ref{likelihood approximation  theorem}.
 
The first example implements the theorem for a mixture of binomials.

\setcounter{foo}{\value{example}}

\setcounter {example} {0}
\begin{example}[continued]  Consider the  data  discussed in \cite{Kupp:Hase:1978} shown in part in Table \ref{example1table1}. Mixture models are  of interest scientifically since the data concerns frequency of implanted foetuses  in   laboratory animals,  and it could be expected that there is underlying clustering. Simple  plots  shows over-dispersion relative to the  variance of a fitted binomial model, which implies that a mixture approach might be appropriate. 
% % Requires the booktabs if the memoir class is not being used
%
 
 \begin{table*}
\caption{Observed frequencies of number of dead implants}
\label{example1table1}
  \begin{tabular}{|c|cccccccc|} \hline
Number of dead implants & 0 & 1 & 2 & 3 & 4 & 5 & 6 & 7 \\
Frequency & 214 &  154 &   83  & 34  & 25 &    9  &  5  & 0\\ \hline
         \end{tabular}
\end{table*}

 Using the polygonal approximation approach allows us to compute easily a good approximation to the mixture. The result can be shown  in Fig.~\ref{example21ig2.pdf}. The crosses show  the fitted model with circles  the data, here with a mixture over $Bin(\pi, 7)$. We also see the mixing proportions and the directional derivative. 
 
   \begin{figure}[htbp] %  figure placement: here, top, bottom, or page
    \centering
   \includegraphics[width=3in]{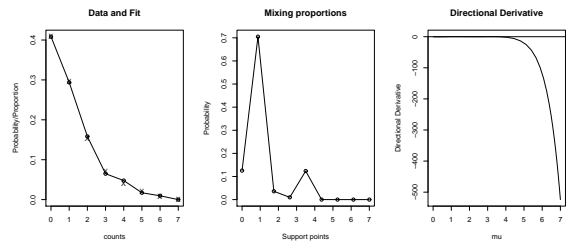}
    \caption{The mixture fit using polygonal approximation}
    \label{example21ig2.pdf}
 \end{figure}

%\cite{Kupp:Hase:1978} used the so-called correlated binomial model for this data, this being an example of a local mixture model as described in 
%\cite{Marr:on:2002}  and \cite{Anaya:Marriott:local:2007}. 

Note in this example the near perfect fit of the data with the mixture model. In terms of the simplex this is easily explained since the maximum likelihood estimate in the simplex, in this case, lies inside the convex hull of the binomial model.  
\end{example}

\setcounter{example}{\value{foo}}

\setcounter{foo}{\value{example}}
\setcounter {example} {3}
\begin{example}[continued]  
For this example, the  distribution of the random variables at all the observed  nodes lies in the $2^3-1= 7$ dimensional simplex, parameterized by the joint probabilities for $(X_1,X_2, X_3)$. If  $H$ were observed each node  would be independent, so that conditionally on $H$ this  space is $3$-dimensional, 
and can be parameterized by the marginal probabilities.  It is easy to show that the conditional model  includes all  $8$ vertices of the $7$ simplex, intersects  six pairs of opposite edges and three pairs of opposite 2-faces.  The full  tripod model is a two component mixture over the three-dimensional full exponential family. Unlike the full convex hull of Example \ref{binomial example},  the two component mixture model need not be  convex in the $-1$-affine space and so can have a complex multimodal likelihood structure.  In order to aid visualisation,  we also consider here the corresponding bipod model, see Fig.~\ref{bipodexamplefig1.eps}

  \begin{figure}[h] %  figure placement: here, top, bottom, or page
    \centering
   \includegraphics[width=2in]{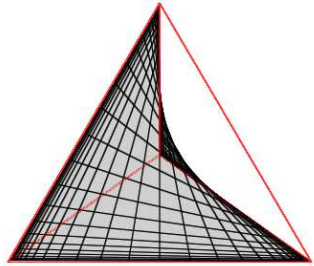}
    \caption{The bipod model: space of unmixed independent distributions showing the ruled-surface structure.}
    \label{bipodexamplefig1.eps}
 \end{figure}

   In the tri- and bi-pod examples,  the unmixed model can be approximated  with unions of $-1$-affine polytopes. These can then be  used to compute likelihood objects on the two hull and convex hull very efficiently just using convex programming. On each polytope the likelihood has a unique maximum 
which may, or may not, be on its boundary. To see the whole two-hull structure,  you just need to glue together this finite number of polytopes and their maxima. Local maxima in the likelihood correspond to internal maxima in the polytopes. 

To see how to construct these approximating polytopes, consider Figure  \ref{bipodexamplefig1.eps}.  The curved surface shown is a, so-called, ruled-surface intersecting the boundary in two pairs of opposite edges. Choose a finite number of support points on each edge of the surface and the same  number on  the opposed  edge. Joining corresponding pairs of points gives a set of $-1$ convex sets, or slices, close to the surface. Any point in the two hull -- that is a convex combination of two  points -- lies in the convex polytope which is the convex hull of two of these slices.

\end{example}

\setcounter{example}{\value{foo}}

\section{Discussion and further work}\label{Discussion and further work}

This paper focused on four main objectives: (1) it  showed that  extended multinomial distributions can be used to construct a computational framework demonstrating commonality between the distinct areas of   information geometry, mixture geometry  and the geometry of graphical models,  (2) it showed how this  structure allow numerically implementation of  results from these areas,   (3) it extended results of information geometry to a simplicial based geometry  for models which do not have a fixed dimension or support, and finally (4) it began the process of   building a computational framework which will act as a proxy for the `space of all distributions'. 

 In  continuous examples,  a compactness condition  is used to keep the underlying geometry finite.  A following paper will look at the case where the compactness condition is not needed. In this case,  infinite dimensional simplexes,  and their closures, are used as the  `space of all distributions',  the extension of classical information geometry here requiring  careful  consideration of convergence, not required here due to finiteness.

Later work will discuss a variety of statistical inference problems -- including model selection and model uncertainty -- using both these finite and infinite frameworks.

\section*{Acknowledgement}
The authors  gratefully thank EPSRC  for the support of  Grant Number EP/E017878/.

%\section*{Supplementary material}
%Further material such as technical details, extended proofs, code, or additional  simulations, figures and examples may appear online, and should be briefly mentioned at appropriate points of your paper.  Please submit any such content as a PDF file along with your paper, entitled `Supplementary Material for Title-of-paper'.  After the acknowledgements, include a section `Supplementary material' in your paper, with the sentence `Supplementary material available at \Bka\ online includes $\ldots$', giving a brief indication of what is to be made available.  If the paper is accepted, the Supplementary Material will need to be prepared using \Bka\ style.

\appendix

\section*{Appendix 1: On the spectral decomposition  of the  Fisher information}\label{Appendix on spectral decomposition  of the  Fisher information}

For notational convenience denote $\pi_{(0)}= (\pi_1, \dots, \pi_k)^T$ so that $1_k^T \pi_{(0)} = 1 - \pi_0$ (i.e. bin $0$ is omitted in the $k\times 1$ vector $\pi_{(0)}$) and $\Pi_{(0)} = diag(\pi_{(0)})$. Without loss, after permutation, assume
$\pi_1 \ge \dots \ge \pi_k$.  Apart from the trivial case $\pi_0 = 1$, when $I(\pi):=\Pi_{(0)}- \pi_{(0)}\pi_{(0)}^T$ vanishes, its spectral decomposition (SpD)   comes in the following cases.
\begin{enumerate}
\item[Case 1] $\pi_l > \pi_{l+1}= \dots \pi_k=0$ for some $0< l < k$. 
 The SpD of $$I(\pi) =
 \left(  
 \begin{array}{c|c}
 \Pi_+ - \pi_+ \pi_+^T & 0 \\ \hline
 0 & 0
 \end{array}
 \right),
$$where $\Pi_+= diag( \pi_+  )$ and $\pi_+= ( \pi_1, \dots, \pi_l)^T$, follows from that of $ \Pi_+ - \pi_+ \pi_+^T $ given below.

\item[Case 2]  $k=1$ is trivial.

\item[Case 3] $k > 1 , \pi= \lambda 1_k, \lambda >0$. The  SpD of $I(\pi)$ is
$$ \lambda C_k + \lambda (1- k\lambda) J_k$$ where
$C_k = I_k -J_k$ and $J_k= k^{-1} 1_k 1_k^T$.  Here $\lambda$ has multiplicity $k-1$ and eigen-space $ [Span(1_k)]^\perp$, while $\tilde \lambda:= \lambda(1-k\lambda)$ has multiplicity $1$ and eigen-space$=Span(1_k)$.
%In particular, since $1-\pi_0=k \lambda$ it follows that  $$\pi_0=0 \iff 1^T \pi=1\iff I(\pi) {\rm \; is\; double\; centred\;},$$ which implies that $I(\pi)= k^{-1} C_k+ 0 J_k$ is singular.
In particular, using $(1-k\lambda) = \pi_0$,
$$I(\pi) {\rm \;is \;singular\;}  \iff \pi_0 = 0.$$

\item[Case 4] This is the generic case. Denoting by $O_m$ the zero matrix of order $m\times m$, and by $P(\nu)$ the rank one orthogonal projector onto Span($\nu$),  $(\nu \ne 0)$, if $\pi_{(0)}= (\lambda_1 1^T_{m_1}| \dots | \lambda_g 1^T_{m_g})^T$ ,
$g >1$ and $\lambda_1 > \dots > \lambda_g >0$, then the SpD is
$$
\sum_{i=1, m_i>1}^g \lambda_i diag(O_{m_{i-}}, C_{m_i}, O_{m_{i-}})
+ \sum_{i=1}^g\tilde \lambda_i P\left(   \left( \frac{  \lambda_1}{  \tilde \lambda_i - \lambda_1 } 1_{m_1}|  \dots |   \frac{  \lambda_g}{  \tilde \lambda_i - \lambda_g } 1_{m_g}    \right)^T \right), 
$$ where  $m_{i-}=\sum \{m_j | j < i \}$,  $m_{i+}=\sum \{m_j | j > i \}$ and the $\tilde \lambda_i$ are the zeros of
$$h(\tilde \lambda) := 1+ \sum_{i=1}^g \frac{m_i \lambda_i^2}{\tilde \lambda -\lambda_i} =
(1- \sum_{i=1}^g m_i \lambda_i) + \tilde \lambda \left(\sum_{i=1}^g \frac{m_i \lambda_i  }{\tilde \lambda -\lambda_i} \right).$$
In particular,  $\{ \tilde \lambda_i : i=1, \cdots, g \}$   are simple eigenvalues satisfying (\ref{interleaving result}) while, whenever $m_i >1$, $\lambda_i$ is also an eigenvalue having multiplicity $m_{i-1}$.
Further, expanding $\det(I(\pi))$, we again find:
$$I(\pi) {\rm \;is \;singular\;}  \iff \pi_0 = 0,$$
so that $\widetilde{\lambda}_{g}>0\Leftrightarrow\pi_{0}>0$, as
claimed. Finally, we note that each $\widetilde{\lambda}_{i}$ $(i<g)$\ is
typically (much) closer to $\lambda_{i}$ than to $\lambda_{i+1}$. For,
considering the graph of $x\rightarrow1/x$, $h\left(  (\lambda_{i}%
+\lambda_{i+1})/2+\delta(\lambda_{i}-\lambda_{i+1})/2\right)  $ $(-1<\delta
<+1)$ is well-approximated by
\[
1-\frac{2m_{i}\lambda_{i}^{2}}{(\lambda_{i}-\lambda_{i+1})(1-\delta)}%
+\frac{2m_{i+1}\lambda_{i+1}^{2}}{(\lambda_{i}-\lambda_{i+1})(1+\delta)}%
\]
whose unique zero $\delta_{\ast}$ over $(-1,1)$\ is positive whenever, as will
typically be the case, $m_{i}=m_{i+1}$ (both will usually be $1$) while
$(m_{i}\lambda_{i}+m_{i+1}\lambda_{i+1})<1/2$. Indeed, a straightforward
analysis shows that, for any $m_{i}$ and $m_{i+1}$, $\delta_{\ast}%
=1+O(\lambda_{i})$ as $\lambda_{i}\rightarrow0$.
\end{enumerate}

\section*{Appendix 2: Proofs} \label{Appendix 2}

%\subsection*{Technical Proofs}

\begin{proof}[Proof of Theorem  \ref{shape of likelihood}]    (a) Immediate.\\

(b)  Let $v \in V_{\mathrm{mix}}$ so that $\sum
v_{i}=0$ and write $v$ as $x+y$ where
\[
x_{i}= \left\{
\begin{array}
[c]{cl}%
v_{i} & \mathrm{if\;} i \in\mathcal{Z} \backslash\{ k^{*}\}\\
0 & \mathrm{if\;} i \in\mathcal{P}\\
\ -\sum_{i \in\mathcal{Z} \backslash\{ k^{*}\} } v_{i} & \mathrm{if\;} i
=k^{*}%
\end{array}
\right.
\]
and
\[
y_{i}= \left\{
\begin{array}
[c]{cl}%
0 & \mathrm{if\;} i \in\mathcal{Z} \backslash\{ k^{*}\}\\
v_{i} & \mathrm{if\;} i \in\mathcal{P}\\
\ v_{k}^{*}+ \sum_{i \in\mathcal{Z} \backslash\{ k^{*}\} } v_{i} &
\mathrm{if\;} i =k^{*}%
\end{array}
\right.
\]
Then, it is immediate that $x$ is in $V^{0}$ and $y$ is in $V^{k^{*}}$, the
decomposition $v = x + y$ being clearly unique.
\end{proof}

\begin{proof} [Proof of Theorem  \ref {general information loss}]   Let $\{ B_k \}_{k=0}^K$ be any finite  measurable partition of ${\cal X}$. Then defining 
$\pi_k(\theta) := \int_{B_k} f(x; \theta) dx$ gives for $i=1, \dots , N,$ and $k=0, \dots, K,$
$$
\pi_{k(i)}(\theta) =\int_{B_{k(i)} } \left\{ f(x_i ; \theta) + (x-x_i)^T \frac{ \partial f }{  \partial x }(x^*_i;\theta) \right\}     dx
$$ where $x^*_i$ is a convex combination of $x$ and $x_i$, \cite{Apost;Math:1965} p.~124, Thm 6--22, and $x_i \in B_{k(i)}$. Thus,
\begin{eqnarray}
\left| \pi_{k(i)} (\theta) - f(x_i; \theta) |B_{k(i)}|  \right| &=&    \left|  \int_{B_{k(i)}}  (x- x_i)^T \frac{ \partial f }{  \partial x }(x^*_i;\theta) dx  \right|  \nonumber \\
&\le & M \,{\rm diam}(B_{k(i)}) |B_{k(i)}|,  \label{density bound} 
\end{eqnarray} where $|B|:= \int_B dx$ and ${\rm diam}(B):= \sup_{(x, y) \in B^2} \| x-y \|$.

It is clear that for compact ${\cal X}$ there exists a sequence of finite measurable  partitions 
${\cal B}(\delta) = \left\{ B_k(\delta) \right\}_{k=0}^{K(\delta)}$ such that as $\delta \rightarrow 0_+$
\begin{equation}\label{shrinking partition}
\max \left|  B_k(\delta) \right| \rightarrow 0, \max \left\{ {\rm diam}(B_k(\delta) \right\} \rightarrow 0 
\end{equation} From (\ref{density bound}) it follows that 
$$\pi_{k(i)}(\theta)  = f(x_i; \theta)\left| B_{k(i)} (\delta) \right|+ o( \left| B_{k(i)} (\delta) \right|   ),     $$ so that
$$
\frac{ Lik_d(\theta)   }{ Lik_d(\theta_0)    } = \frac{ Lik_c(\theta)   }{ Lik_c(\theta_0)    } \left\{   \frac{  \prod_{i=1}^N \left(  1+ \frac{o( \left| B_{k(i)} (\delta) \right|   ) }{   f(x_i; \theta_0)\left| B_{k(i)} (\delta) \right|}   \right)  }{  \prod_{i=1}^N \left(  1+ \frac{o( \left| B_{k(i)} (\delta) \right|   ) }{   f(x_i; \theta)\left| B_{k(i)} (\delta) \right|}   \right)    }  \right\} .
$$
Since $f(x_i;\theta)$ is bounded away from zero for all $\theta$, this gives
\begin{eqnarray*}
\frac{ Lik_d(\theta)   }{ Lik_d(\theta_0)    } &  = & \frac{ Lik_c(\theta)  (1+ O(\delta) )  }{ Lik_c(\theta_0)   (1+ O(\delta))  } \\
&=&  \frac{ Lik_c(\theta)   }{ Lik_c(\theta_0)   }  (1+ O(\delta)),
\end{eqnarray*} from which the result follows. 

\end{proof}

\begin{proof}[Proof  of Theorem  \ref{exponential information loss theorem} ] 
From the uniform continuity of $s(x)$ and the compactness of $s({\cal X})$, there exists a finite measurable partition $\{B_k\}_{k=0}^{K(\epsilon)}$ such that for all $k$ and for all $x, y \in B_k$,
\begin{equation}\label{bound on s(x)}
\| s(x) - s(y) \| \le \epsilon.
\end{equation} It follows from (\ref{bound on s(x)}) that for all $x \in B_k$ and for all $\theta \in \Theta$,
\begin{equation}\label{bound on mean of s(x)}
\| s(x) - E_\theta(s(X))| X \in B_k)  \|  \le \epsilon.
\end{equation}  From (\ref{bound on mean of s(x)})  it further follows that 
\begin{equation}\label{bound on covariance}
Cov_\theta(s_r(x), s_s(x) | X \in B_k) = O(\epsilon^2)
\end{equation} and
\begin{equation}\label{bound on shewness}
T_{rst}(\theta |k) := E_\theta( t_r t_s t_t  | X \in B_k) = O(\epsilon^3)
\end{equation} where $ t_r := s_r- E(s_r(X) | X \in B_k)$.

Further by direct calculation it follows that 
\begin{eqnarray}
\frac{\partial}{\partial \theta_r} \log \pi_k(\theta) & = &E_\theta(s_r(X) | X \in B_k)- \frac{\partial \psi }{\partial \theta_r} (\theta)  \label{direct calc 1} \\
\frac{\partial^2}{\partial \theta_r\partial \theta_s} \log \pi_k(\theta) & = &Cov_\theta(s_r(X) , s_s(X) | X \in B_k)- \frac{\partial^2 \psi }{\partial \theta_r\partial \theta_s} (\theta)  \label{direct calc 2} \\
\frac{\partial^2}{\partial \theta_r \partial \theta_s \partial \theta_t} \log \pi_k(\theta) & = &E_\theta( t_r t_s t_t  | X \in B_k)  - \frac{\partial^3 \psi }{\partial \theta_r \partial \theta_s \partial \theta_t}(\theta)   \label{direct calc 3}.
\end{eqnarray} Finally, (a)  follows immediately from (\ref{bound on s(x)}) and (\ref{bound on mean of s(x)}),
(b)   from (\ref{bound on covariance}) and (\ref{direct calc 2}), and 
(c)  from (\ref{bound on shewness}) and (\ref{direct calc 3}).
\end{proof}

\begin{proof}[Proof  of Corollary  \ref{corollary for exponential family}] 
The score equations for $\hat \theta_c$ are   $\frac{\partial \psi}{\partial \theta_k} ( \hat \theta_c)= \frac{ \sum s_k(x_i) }{N}$, while from (\ref{direct calc 1}) those for $\hat \theta_d$ are
$$
\frac{\partial \psi}{\partial \theta_k} ( \hat \theta_d)= \frac{ \sum n_k E(s_k(X) | X \in B_k)  }{N}. 
$$
Using (\ref{bound on covariance}) and that $\psi^\prime$ has a continuous inverse gives (\ref{difference in MLE}), while (\ref{difference in information }) follows from  (\ref{difference in MLE}) and (\ref{direct calc 2}).

\end{proof}

\begin{proof}[Proof  of Theorem  \ref{lindsay embedding theorem}]  
(a) The log-likelihood can be written as 
 $\sum_{i \in {\mathcal P}} n_i \log \pi_i $ which is clearly constant for all probability vectors with the same image under $\Pi_L$ since they share  the same elements $\pi_i, i \in {\mathcal P}$.
 (b) Since $\Pi_L$     is linear it preserves $-1$ convexity.
\end{proof}

\begin{proof}[Proof  of Theorem  \ref{total positivity}] 
%Consider the matrix 
%$B(\theta_0, \dots, \theta_k)$ whose $i,j$ element is given by $$\pi_i(\theta_j):=\pi_i \exp\left[  s_i \theta_j - \psi(\theta_j) \right].$$ From \cite{Hous:Theory:1975}, page 33
%$$
%Rank(B-b_0 1^T)=Rank(B)-1
%$$ for any non-zero vector $b_0$.
% Assume $s_1 < \dots < s_k$
% $$
% Rank(B)= Rank( \Pi (\exp(s_i \theta_j)) \Psi) 
% $$ where $\Pi=diag(\pi), \Psi=diag(\exp(-\psi(\theta_j)))$.
%So it follows that $$Rank(B) =  Rank (\exp(s_i \theta_j)     )$$
%
%To prove the result is is just necessary to show that $B$ has rank $k$ in the $k$ dimensional simplex. This follows by using the strict positive definiteness which is a property of the exponential family, see Example A.11,  page 490  \cite{Marsh;Olk;1979}. In particular,  $B^* := (\exp(s_i \theta_j)     )$ is a generalised Vandermonde matrix and is a strictly positive definite if  
% $$ \theta_0  < \dots  < \theta_k \; {\rm and \;} s_0 < \dots < s_k, $$
%  \cite{Marsh;Olk;1979} page 487 and 490 and  \cite{Karlin;Total;1968}.   
%  
%  
%  

  For any $(\pi_{i})\in\Delta^{k}$ with each $\pi_{i}>0$, $\theta
_{0}< \cdots <\theta_{k}$ and $s_{0}< \cdots <s_{k}$, let $B=(\pi(\theta_{0}),...,\pi(\theta_{k}))$ have general element $$\pi_{i}(\theta_{j}):=\pi_{i}\exp[s_{i}\theta_{j}-\psi(\theta_{j})].$$ Further, let $\widetilde{B}=B-\pi(\theta_{0})1_{k+1}^{T}$, whose general column is $\pi(\theta_{j})-\pi(\theta_{0})$. Then, it suffices to show that $\widetilde{B}$ has rank
$k$. But, using  \cite{Hous:Theory:1975} p.33, $Rank(\widetilde{B})=Rank(B)-1$, so that
\[
Rank(\widetilde{B})=k\Leftrightarrow B\text{ is nonsingular }\Leftrightarrow
B^{\ast}\text{ is nonsingular,}%
\]
where $B^*=(\exp[s_{i}\theta_{j}])$. It suffices, then, to recall  \cite{Karlin;Total;1968}
that $K(x,y)=\exp(xy)$ is strictly total positivity (of order\ $\infty$), so
that $\det B^*>0$.

 \end{proof}

\begin{proof}[Proof  of Theorem \ref{likelihood approximation  theorem}]

   We use a similar expansion to (\ref{likelihood expansion turning point}), adapted to take into
account the fact that the NPMLE is defined by directional derivatives being non-negative, rather than zero \cite{Lind:mixt:1995}.

If $\pi$ is a member of the convex hull of $\pi(\theta),$ then the directional derivative from $\hat \pi^{NP}$ to $\pi$ is a finite convex combination of 
elements of the convex cone of  directional derivatives from $\hat \pi^{NP}$ to points in the curve $\pi(\theta)$.
For any point $\pi(\theta)$ consider the  perturbation from $\hat \pi^{NP}$ of the form 
$$\pi(\lambda) := \hat \pi^{NP} + \lambda(\pi(\theta) - \hat \pi^{NP}).$$  There are two cases to consider: (i) either $\theta$ is a support  point of $ \hat \pi^{NP}$ or  (ii) it is not.

\noindent Case (i)  In this case the directional derivative are zero. Accordingly  we can apply (\ref{likelihood expansion turning point}) directly to have that 
the change in log-likelihood is $o(\epsilon)$.

\noindent Case (ii) In this case, for small enough positive $\lambda$,   $\pi(\lambda)$ remains in the convex hull. Further, the difference in log-likelihood values is then $$\sum_{i | n_i > 0}  n_i \log(\pi_i(\lambda)) - \sum_{i | n_i > 0} n_i \log (\hat \pi^{NP}_i).$$ Since the directional derivatives are now  non-zero, consider  the first order term in the Taylor expansion of this difference:
\begin{eqnarray*}
\lambda \left. \left\{ \sum_{i | n_i > 0} \frac{n_i  (\pi_i(\theta) - \hat \pi_i^{NP}) }{ \pi_i(\lambda)}\right\} \right|_{\lambda=0} & = & \lambda \sum_{i=0}^k \frac{n_i  (\pi_i(\theta) - \hat \pi_i^{NP}) }{\hat \pi_i^{NP}} \\
&=& \lambda N \sum_{i=0}^k  \frac{(\hat  \pi^{G}_i -  \hat \pi_i^{NP}  )  (\pi_i(\theta) - \hat \pi_i^{NP} ) }{ \hat \pi_i^{NP}} \\
&=& \lambda N \left<  (\hat  \pi^{G} - \hat \pi^{NP}   ) ,  (\pi(\theta) - \hat \pi^{NP}) \right>_{\hat \pi^{NP}} \\
& \le &  \lambda N  || ( \hat \pi^{G}-  \hat \pi^{NP})      ||_{\hat \pi^{NP}} ||  (\pi(\theta) - \hat \pi^{NP})    ||_{\hat \pi^{NP}}. 
\end{eqnarray*}
Considering  $\lambda$  small enough that 
$$  || \pi(\lambda)-\hat \pi ^{NP}  || = \lambda ||  (\pi(\theta) -\hat \pi ^{NP})    ||_{\hat \pi ^{NP}}  \le \epsilon,$$
we have that to first order the change in log-likelihood values for points  $\pi(\lambda)$  within $\epsilon$ of ${\hat \pi ^{NP}} $ is bounded by
$$
 \epsilon N  || ( \hat \pi^{G}-  {\hat \pi ^{NP}} )     ||_{\hat \pi ^{NP}}.
$$
So it has been shown that  all points in the convex hull of $\pi(\theta)$ which are within $\epsilon$ of $\hat \pi^{NP}$ satisfy (\ref{likelihood bound}). From Lemma \ref{lemma on uniform distance} there is at least one point in the convex hull of the polygon which is within $\epsilon$ of the convex hull. Hence the maximum likelihood value at $\hat \pi$ also satisfies  (\ref{likelihood bound}). 
\end{proof}

\bibliographystyle{plain} 
\bibliography{refsfound}

\end{document}